\documentclass[a4paper,10pt]{article}

    \usepackage{float}

    \usepackage[english]{babel}
    \usepackage{amsthm}
    \usepackage[intlimits]{amsmath}
    \usepackage{amssymb}
    \usepackage{inputenc}
    \usepackage{setspace}
    \usepackage{enumerate}
    \usepackage{esint}

    \usepackage{hyperref}

    \usepackage{dsfont}
    \usepackage{array}
    \usepackage{graphicx}
    \usepackage[babel]{csquotes}
    \usepackage{geometry}
    \usepackage{bbm}
    \usepackage{stmaryrd}
    \usepackage[all]{xy}
    \usepackage{mathrsfs}

    \theoremstyle{plain}
    \newtheorem{thm}{Theorem}[section]
    \newtheorem{prop}[thm]{Proposition}
    
    \newtheorem{lem}[thm]{Lemma}
    
    \theoremstyle{definition}
    \newtheorem{defn}[thm]{Definition}
    
    \newtheorem{rem}[thm]{Remark}

    \def \B{\mathscr{B}}
    \def \L{\mathscr{L}}
    
    \def \R{\mathbb{R}}
    \def \Z{\mathbb{Z}}
    \def \sph{\mathbb{S}}
    \newcommand {\sign} {\mathop \mathrm{sign}} 
    
    \def \ae{\text{a.e.\ }}
    \def \torus{{[0,\Lambda)^d}}
    
    \def \numphases{P}
    \newcommand {\supp} {\mathop \textup{supp}}
    \def \chara{\mathbf{1}}
    \def \eps{\varepsilon}
    \def \exc{\mathcal{E}}
    \DeclareMathOperator*{\esssup}{ess\,sup}

    \title{Convergence of the Allen-Cahn Equation to multi-phase mean-curvature flow}
    \author{Tim Laux\footnote{Max-Planck-Institut f\"ur Mathematik in den Naturwissenschaften, Inselstra{\ss}e 22, 04103 Leipzig, Germany. Please use {tim.laux@mis.mpg.de} for correspondence.  } 
    \and Thilo Simon\footnotemark[1]}
    \date{\today}
    \usepackage[backend=bibtex,maxnames=4]{biblatex}
    \bibliography{lit.bib}
    \DefineBibliographyStrings{english}{%
      references = {Bibliography},
    }

\begin{document}

      \maketitle
      \begin{abstract}
      We present a convergence result for solutions of the vector-valued Allen-Cahn Equation. In the spirit of the work of Luckhaus and Sturzenhecker we establish convergence towards a distributional formulation 
      of multi-phase mean-curvature flow using sets of finite perimeter. 
      Like their result, ours relies on the assumption that the time-integrated energies of the approximations converge to those of the limit.
      Furthermore, we apply our proof to two variants of the equation, incorporating external forces and a volume constraint.
    %

    \medskip

    \noindent \textbf{Keywords:} Mean curvature flow, Allen-Cahn Equation

    \medskip

    \noindent \textbf{Mathematical Subject Classification:} 35A15, 35K57, 35K93, 74N20
    \end{abstract}

    \section*{Introduction}
    Motion by mean curvature is an important geometric evolution equation and arises in various problems in geometry, physics and the sciences. 
    Its multi-phase version for example is a popular model for the evolution of grain boundaries in polycrystals undergoing heat treatment, 
    already motivated in \cite{mullins1956two}.
    The Allen-Cahn Equation
    \begin{align}\label{allen cahn}
	\partial_t u_\eps = \Delta u_\eps -\frac{1}{\eps^2} \partial_u W (u_\eps)
      \end{align}
      is a well-established phase-field approximation for (multi-phase) mean-curvature flow \cite{allen1979microscopic}, 
    replacing sharp interfaces by diffused transition layers.

    \medskip

    The derivation of motion by mean curvature as the singular limit of the Allen-Cahn Equation has a 
    long history and is well-understood in the two-phase case: The first formal asymptotic expansions were constructed by Rubinstein, Sternberg and Keller 
    \cite{rubinstein1989fast}.
    Convergence for a smooth evolution was proved independently by De Mottoni and Schatzman \cite{de1995geometrical} and Chen \cite{chen1992generation}.
    Bronsard and Kohn \cite{bronsard1991motion} used the gradient flow structure of \eqref{allen cahn} 
    to prove compactness, and, in the radially symmetric case, convergence to motion by mean curvature. 
    For the long-time behavior past singularities the following two well-established notions of weak solutions have 
    proven to be useful for understanding the singular limit of \eqref{allen cahn}: viscosity solutions \cite{evans1991motion,chen1991uniqueness}
    and Brakke's varifold-solutions \cite{brakke1978motion}.

    Viscosity solutions on the one hand are based on the level-set formulation \cite{osher1988fronts} and the well-known geometric comparison principle of two-phase
    mean-curvature flow. Evans, Soner and Souganidis \cite{evans1992phase} rigorously proved the convergence
    towards the viscosity solution -- at least if the level-set of the viscosity solution does not develop an interior but remains ``thin''. Barles, Soner and Souganidis
    \cite{barles1993front} showed in particular that this holds true for mean-convex or star shaped initial conditions. 

    Brakke's varifold-solutions \cite{brakke1978motion} on the other hand are based on the gradient flow structure of mean-curvature flow and are defined by the optimal dissipation of energy along the
    solution. 
    Ilmanen proved convergence towards Brakke's formulation
    \cite{ilmanen1993convergence} in the two-phase case by translating Huisken's celebrated monotonicity formula \cite{huisken1990asymptotic} to the phase-field framework of \eqref{allen cahn}.

    While the question of convergence of the Allen-Cahn Equation \eqref{allen cahn} seems to be almost settled in the two-phase case, little is known in the multi-phase case.
    Even the work of Ilmanen \cite{ilmanen1993convergence} seems not to apply since he makes use of comparison techniques at a crucial point. 
    Bronsard and Reitich \cite{bronsard1993three} carried out a formal asymptotic expansion at a triple junction and proved short-time existence. 
    However, to the best of our knowledge,  rigorous long-time convergence results past singularities have not been available
    prior to the present work.

    \medskip
    
    In comparison to its two-phase counterpart, \emph{multi-phase} mean-curvature flow is still poorly understood.
    The analytic study of the planar case started with the work of Mantegazza, Novaga and Tortorelli \cite{mantegazza2003motion} 
    who studied the evolution of a single triple junction.
    Recently Mantegazza, Novaga, Pluda and Schulze \cite{mantegazzaevolution} extended these results to the case of two triple junctions.
    Ilmanen, Neves and Schulze \cite{ilmanen2014short} proved short-time existence even when starting from certain non-regular networks, which should 
    allow to continue the flow through all generic/stable singularities that form during the evolution of a planar network.
    Only recently, global weak solutions were constructed in the substantial work of Kim and Tonegawa \cite{kim2015mean}.
    They proved convergence of a variant of Brakke's original scheme towards a \emph{non-trivial} Brakke flow.
    Uniqueness of the evolution is still unclear but is expected in generic situations.
    
    \medskip

    Our proof is of variational nature in the sense that it is based on the gradient flow structure of the Allen-Cahn Equation and mean-curvature
    flow. In particular, we use some techniques known from the analytical study of the static
    analogue of \eqref{allen cahn}, initiated by the work of Modica and Mortola \cite{modica1977esempio}. Modica \cite{modica1987gradient}
    and Sternberg \cite{sternberg1988effect} provided the convergence of the Ginzburg-Landau Energy (see \eqref{E eps} for a definition in the multi-phase case) 
    towards a multiple of the perimeter functional in the sense of $\Gamma$-convergence. 
    Kohn and Sternberg \cite{kohn1989local} were able to construct \emph{local} minimizers of the Ginzburg-Landau Energy \eqref{E eps} based on the above $\Gamma$-convergence.
    Furthermore, it turns out that the convergence of the Ginzburg-Landau Energy towards the perimeter functional is even stronger:
    Luckhaus and Modica \cite{luckhaus1989gibbs} proved that also
    the first variations of the energies converge towards the mean curvature -- the first variation of the perimeter functional -- 
    by the clever use of a classical argument of Reshetnyak \cite{reshetnyak1968weak}.
    A year later, Baldo \cite{baldo1990minimal} extended the $\Gamma$-convergence of the energies \eqref{E eps} to the multi-phase case.

    \medskip
    
    However, the $\Gamma$-convergence of the energies does not imply the convergence of the according gradient flows. 
    Since every gradient flow comes with a metric, it is evident that one needs conditions on both, the metric tensor and the energy, to verify the convergence.
    Sandier and Serfaty \cite{sandier2004gamma} provided sufficient conditions for this convergence.
    Serfaty \cite{serfaty2011gamma} has already mentioned that these assumptions are guaranteed by the
    works of R\"oger and Sch\"atzle \cite{roger2006modified} on the Willmore functional and
    Mugnai and R\"oger \cite{mugnai2008allen} on the action functional of the Allen-Cahn Equation.
    This result is restricted to two-phase mean-curvature flow in dimensions $d\leq 3$ though.
    
    \medskip
    
    From a conceptional viewpoint, our proof is closely related to a number of other convergence proofs for implicit time-discretizations
    in the spirit of De Giorgi's \emph{minimizing movements} \cite{de1993new}.
    Luckhaus and Sturzenhecker \cite{LucStu95} established the convergence of the time-discretization
    proposed by Almgren, Taylor and Wang \cite{ATW93}; and Luckhaus and Sturzenhecker \cite{LucStu95} towards a 
    distributional solution of mean-curvature flow, see \eqref{H=v} and \eqref{v=dtX} for a multi-phase version of this formulation.
    Recently, Otto and the first author \cite{laux2015convergence} proved convergence of the thresholding scheme of Merriman, Bence and Osher \cite{MBO92,MBO94}
    in the multi-phase case based on the minimizing movements interpretation of Esedo\u{g}lu and Otto \cite{EseOtt14}.
    Over the last decades, this variational viewpoint has proven to be flexible enough to study a tremendous amount of problems such as
    the  Stefan Problem \cite{LucStu95} and its anisotropic variant \cite{garcke2011existence}, the
    Mullins-Sekerka Flow \cite{roger2005existence} and its multi-phase variant \cite{bronsard1998multi},
    volume-preserving mean-curvature flow \cite{mugnai2015global,LauSwa15}, 
    the evolution of martensitic phase transitions \cite{dondl2010sharp}, and many more.

    \medskip
    
    Our main result, Theorem \ref{thm AC2MCF}, establishes the convergence of solutions of \eqref{allen cahn} 
    for a general class of potentials and any space dimension.
    Like the results of Luckhaus and Sturzenhecker \cite{LucStu95}, and Felix Otto and the first author \cite{laux2015convergence},
    also ours is only a \emph{conditional} convergence result in the sense that we assume the time-integrated energy of the approximations
    to converge to the time-integrated  energy of the limit, see \eqref{conv_ass}.
    Although this is a very natural assumption, it is not guaranteed by the a priori estimates coming from the energy-dissipation equality \eqref{energy-dissipation equality}.
    However, the verification of this assumption is non-trivial and even fails for certain initial data, cf.\ \cite{bronsard1996existence} for an example of higher multiplicity
    interfaces in the limit of the volume-preserving Allen-Cahn Equation.
    
    The main idea of our proof is to multiply the Allen-Cahn Equation
    $
	\partial_t u_\eps = \Delta u_\eps -\frac{1}{\eps^2} \partial_u W (u_\eps)
      $
    with $\eps \left( \xi \cdot \nabla \right) u_\eps$, integrate in space and time  and pass to the limit $\eps\downarrow0$.
    To this end we extend the above mentioned argument of Luckhaus and Modica \cite{luckhaus1989gibbs} to the  multi-phase case and obtain
    the curvature-term $\int_\Sigma H\, \xi \cdot \nu$ from the right-hand side.
    The more delicate part, and the core of this paper, is how to pass to the limit in the velocity-term $\int_\Sigma V\, \xi \cdot \nu$.
    The difficulty is that one has to pass to the limit in a \emph{product} of weakly converging terms, the normal and the velocity.
    We overcome this difficulty by ``freezing'' the normal and introducing an appropriate approximation \eqref{exc eps} of the tilt-excess.
    After doing so it turns out that the new nonlinearity with the frozen normal can be written  as a derivative of a compact quantity. 
    The technique of freezing the normal was used before in \cite{laux2015convergence}, where the authors introduce an approximation of the energy-excess.

    To work with the tilt-excess instead of the energy-excess seems very natural to us in this particular problem and might be interesting 
    in other cases too. The only extra difficulty is that one has to pass to the limit in the nonlinear quantity \eqref{exc eps}.
    However, our problem seems to be much simpler than the one in \cite{laux2015convergence} as we do not have to work on multiple time scales.

    \medskip
    
    The structure of the paper is as follows. In Section \ref{sec:result} we introduce the notation and state our main result, Theorem \ref{thm AC2MCF}.
    In Section \ref{sec:comp} we prove compactness of the solutions together with bounds on the normal velocities.
    We took care to be precise in this section but do not claim the originality of the results.
    We use a general chain rule of Ambrosio and Dal Maso \cite{ambrosio1990general} to identify the nonlinearities in the multi-phase 
    case as derivatives. Furthermore, we repeat the application of De Giorgi's structure result from \cite{laux2015convergence} to handle the excess.
    In Section \ref{sec:conv} we pass to the limit in the equation.
    Since this is the most original part, we give a short overview over the idea of the proof first.
    We then present our extension of the Reshetnyak argument by Luckhaus and Modica \cite{luckhaus1989gibbs} in Proposition \ref{multi luckhaus modica} to handle the curvature-term and
    prove the convergence of the velocity-term in Proposition \ref{multi prop dt nu}, which is the main novelty and the core of the paper.
    We conclude the section with the proof of the main result, Theorem \ref{thm AC2MCF}.
    In Section \ref{sec:forces,volume} we apply our method to the cases when external forces are present or a volume-constraint is active, see Theorems \ref{thm FAC2FMCF} and \ref{thm VPAC2VPMCF}. 

     \section{Main results}\label{sec:result}
      The Allen-Cahn Equation \eqref{allen cahn} describes a system of fast reaction and slow diffusion and is the (by the factor $\frac1\eps$ accelerated) 
      $L^2$-gradient flow of the Ginzburg-Landau Energy
      \begin{align}\label{E eps}
	E_\eps(u_\eps) = \int \frac\eps2 \left|\nabla u_\eps\right|^2 + \frac{1}{\eps} W(u_\eps)\,dx.
      \end{align}
      For convenience we will work with periodic boundary conditions for $u$, i.e.\ on the flat torus $\torus$ for some $\Lambda >0$ and write $\int\,dx$ 
      short for $\int_\torus \,dx$.
      
      Here the (unknown) order parameter $u_\eps \colon \R^d \to \R^N$ is vector-valued and $W\colon \R^N \to [0,\infty)$
      is a smooth multi-well potential with finitely many zeros at $u=\alpha_1, \dots ,\alpha_\numphases \in \R^N$.
      We will furthermore impose polynomial growth and convexity of $W$ at infinity:
      \begin{enumerate}
       \item There exist constants $0<c<C < \infty$, $R<\infty$ and an exponent $p\geq2$ such that
	  \begin{equation}\label{growth} 	
	    c |u|^p \leq W(u) \leq C|u|^p \quad \text{for }  |u| \geq R
	  \end{equation}
	  and
	  \begin{equation}\label{growth_derivative}
	    |\partial_u W(u)| \leq C|u|^{p-1}\quad \text{for }  |u| \geq R.
	  \end{equation}
	\item There exist smooth functions $W_{conv}$, $W_{pert} : \R^N \to [0,\infty)$ such that
	  \begin{equation}\label{decomposition_convex}
	   W = W_{conv} + W_{pert}.
	  \end{equation}
	  Here, the function $W_{conv}$ is convex and $W_{pert}$ has at most quadratic growth in the sense that there exists a constant $\tilde C$ such that we have
	  \begin{equation}\label{growth_pertubation}
	    \partial_u^2 W_{pert}(u) \leq \tilde C.
	  \end{equation}
      \end{enumerate}
      These assumptions seem to be very natural to us:
      The classical two-well potential $W(u) = (u^2-1)^2$ for $u \in \R$ clearly has these properties and they are compatible with polynomial potentials also in the case of systems.
           
      By now it is a classical result due to Baldo \cite{baldo1990minimal} that these energies $\Gamma$-converge w.r.t.\ the $L^1$-topology to an 
      \emph{optimal partition energy} given by
      \begin{align}\label{E}
	E(\chi) = \frac12\sum_{1\leq i,j \leq P}\sigma_{ij} \int \frac{1}{2}\left(|\nabla \chi_i|+ |\nabla \chi_j| - |\nabla (\chi_i + \chi_j)|\right),
      \end{align}
      for a partition $\chi_1,\dots,\chi_\numphases \colon \torus \to \{0,1\}$ satisfying the compatibility condition $\sum_{1\leq i\leq P} \chi_i = 1$ a.e.
      Note that for $\chi_i = \chara_{\Omega_i}$ we can also rewrite the limiting energy in terms of the interfaces
      $\Sigma_{ij} := \partial_\ast \Omega_i \cap \partial_\ast \Omega_j$ between the phases, where $\partial_\ast$ denotes the reduced boundary:
      \[E(\chi)= \frac12\sum_{1\leq i,j \leq P}\sigma_{ij} \mathcal{H}^{d-1}\left(\Sigma_{ij}\right).\]
      The link between $u_\eps$ and $\chi$ is given by 
      \[
      u_\eps \to u := \sum_{1\leq j\leq P} \chi_i \alpha_i.
      \]
      
      The constants $\sigma_{ij}$ are the geodesic distances with respect to the metric $2 W(u)\langle \cdot , \cdot \rangle$, i.e.\
      \[
	\sigma_{ij} = d_W(\alpha_i,\alpha_j),
      \]
      where the geodesic distance is defined as
      \begin{equation}\label{geodesic distance}
      d_W(u,v):= \inf \left\{\int_0^1 \sqrt{2W(\gamma)}|\dot \gamma| ds : \gamma \colon [0,1] \to \R^n \text{ a } C^1\text{ curve with } \gamma(0) = u, \gamma(1)=v \right\}.
      \end{equation}
      The surface tensions satisfy the triangle inequality
      \[
	\sigma_{ij} \leq \sigma_{ik} + \sigma_{kj} \quad \text{for all } i,j,k
      \]
      and clearly 
      \[
      \sigma_{ii} =0, \quad \sigma_{ij}>0 \quad\text{for }i\neq j, \quad  \text{and} \quad \sigma_{ij} = \sigma_{ji}. 
      \]
      
      It is an interesting and non-trivial question to find an appropriate potential $W$ which generates given surface tensions $\sigma$.
      In a recent paper, such potentials with multiple wells have been constructed by Bretin and Masnou \cite{bretin2015new} for a related class of energies.
%
%
      We will want to localize both the Ginzburg-Landau Energy and the optimal partition energy.
      Given $\eta \in C(\torus)$ let
      \begin{align*}
      E_\eps(\eta,u_\eps) & := \int \eta \left( \frac{\eps}{2} \left|\nabla u \right|^2 + \frac1\eps W(u)\right)\,dx,\\
      E(\eta ,\chi) & := E(\eta ,u) : = \frac12\sum_{1\leq i,j \leq P}\sigma_{ij} \int \eta \frac{1}{2}\left(|\nabla \chi_i|+ |\nabla \chi_j| - |\nabla (\chi_i + \chi_j)|\right).
      \end{align*}
    
      For our result we will impose
      \begin{equation}\label{conv_ass}
	\int_0^T E_\eps (u_\eps) \,dt \to \int_0^T E(\chi)\,dt
      \end{equation}
      ruling out a certain loss of surface area in the limit $\eps \downarrow 0$.
      Under this assumption we will establish convergence towards the following distributional formulation of mean-curvature flow, see \cite{LucStu95,laux2015convergence}.
      
      \begin{defn}[Motion by mean curvature]\label{def_motion_by_mean_curvature}
	Fix some finite time horizon $T<\infty$, a $P\times P$-matrix of surface tensions $\sigma$ as above and initial data 
	$\chi^0\colon \torus \to \{0,1\}^P$ with $E_0 := E(\chi^0) <\infty$ and $\sum_{1\leq i\leq P} \chi^0_i =1$.
	We say that
	\begin{align*}
	  \chi \in C\Big([0,T];L^2(\torus;\{0,1\}^P)\Big)
	\end{align*}
	with $ \sup_t E(\chi) < \infty$ and $\sum_{1\leq i\leq P} \chi_i =1$
	\emph{moves by mean curvature} if there exist densities $V_i$ with
	\begin{equation}
	  \int_0^T \int V_i^2 \left| \nabla \chi_i \right| dt < \infty
	\end{equation}
	satisfying the following properties:
	\begin{enumerate}
	  \item For all $\xi\in C^\infty_0((0,T)\times \torus, \R^d)$
	    \begin{align}\label{H=v}
	      \sum_{1\leq i,j\leq P} \sigma_{ij} \int_0^T \int \left(\nabla\cdot\xi - \nu_i \cdot \nabla \xi \,\nu_i
	      -  V_i\, \xi \cdot \nu_i  \right) 
	      \frac{1}{2}\left(|\nabla \chi_i|+ |\nabla \chi_j| - |\nabla (\chi_i + \chi_j)|\right) dt= 0,
	    \end{align}
	      where $\nu_i$ is the inner normal of $\chi_i$, i.e.\  the density of $\nabla \chi_i$ with respect to $|\nabla \chi_i|$.
	  \item The functions $V_i$ are the normal velocities of the interfaces in the sense that
	  \begin{equation}\label{v=dtX}
	    \partial_t \chi_i = V_i|\nabla \chi_i | dt \text{\,  distributionally in }(0,T)\times\torus.
	  \end{equation}
	  
	  \item The initial data is achieved in the space $C([0,T];L^2(\torus))$, i.e.\
	  \[\chi_i(0) = \chi_i^0\]
	  in $L^2(\torus)$ for all $1\leq i\leq P$.
	\end{enumerate}
      \end{defn}

      If the evolution is smooth one can integrate by parts and obtain the classical formulation of multi-phase mean-curvature flow
      consisting of the evolution law
      \[
       V_{ij} = H_{ij}\quad \text{on } \Sigma_{ij}
      \]
      together with Herring's well-known angle condition
      \[
       \sum_{i,j} \sigma_{ij} \nu_{ij} =0 \quad \text{at triple junctions}.
      \]
      Comparing to the more general evolution law $V_{ij} = \sigma_{ij} \mu_{ij} H_{ij}$ we see that
      in our case the mobility $\mu_{ij}$ of the interface $\Sigma_{ij}$ is given by $\mu_{ij}=\frac1{\sigma_{ij}}$.
      How to generate general mobilities seems not to be settled yet.
                  
      Our main result is the following theorem.
      \begin{thm}\label{thm AC2MCF}
      Let $W$ satisfy the growth conditions \eqref{growth} and \eqref{growth_derivative}, as well as the convexity at infinity \eqref{decomposition_convex}.
      Let $T<\infty$ be an arbitrary finite time horizon.
      Given a sequence of initial data $u_\eps^{0} \colon \torus \to \R^N$ approximating a partition $\chi^0$, in the sense that
      \begin{equation}\label{initial data}
	  u^{0}_\eps \to \sum_{1\leq i \leq P} \chi_i \alpha_i \quad \text{a.e.\ and}\quad   E_0:= E(\chi^0) = \lim_{\eps\downarrow 0} E_\eps(u^{0}_\eps)<\infty,
      \end{equation}
      there exists a subsequence $\eps\downarrow 0$ such that
      the solutions $u_\eps$ of \eqref{allen cahn} with initial datum $u_\eps^0$ converge to a time-dependent partition $\chi \in C([0,T];L^2(\torus;\{0,1\}^P))$.
      If the convergence assumption \eqref{conv_ass} holds, then $\chi$ moves by mean curvature according to 
      Definition \ref{def_motion_by_mean_curvature}.
      \end{thm}

      \begin{rem}
      For any partition $\chi^0 \in BV\left(\torus;\{0,1\}^P\right)$ it is possible to choose $u_\eps^0$ with $u_\eps^0 \to \sum_{1 \leq i \leq P} \chi_i \alpha_i$ in $L^1$ and $E_\eps(u_\eps^0) \to E_0(\chi)$
      by the $\Gamma$-convergence result \cite{baldo1990minimal}.
      \end{rem}
      
      Using some adjustments of our argument we can also deal with external forces and a volume constraint.
      
      \begin{thm}\label{thm FAC2FMCF}
	Let $W$ satisfy \eqref{growth}, \eqref{growth_derivative} and \eqref{decomposition_convex} and let $T<\infty$ be an arbitrary finite time horizon.
	Given a sequence of initial data $u_\eps^{0} \colon \torus \to \R^N$ approximating a partition $\chi^0$, in the sense of \eqref{initial data}
	and forces $f_\eps: [0,T]\times \torus \to \R^N$ such that
	\[\sup_{\eps>0} \int_0^T \int |f_\eps|^2+|\partial_t f_\eps|^2 + |\nabla f_\eps|^2 dx \,dt < \infty\]
	there exists a subsequence $\eps\downarrow 0$ such that
	the solutions $u_\eps$ of 
	\begin{equation}\label{allen cahn f}
	  \begin{cases}
	    \partial_t u_\eps  = \Delta u_\eps - \frac{1}{\eps^2}\partial_u W(u_\eps) + \frac{1}{\eps}f_\eps & \text{ in } [0,T]\times \torus, \\
	    u_\eps = u_\eps^0 & \text{ on } {0}\times\torus
	  \end{cases}
	\end{equation}
	converge to a time-dependent partition $\chi \in C([0,T];L^2(\torus;\{0,1\}^P))$.
	Furthermore, the forces also have a limit $f_\eps \to f$ in $L^2$.
	If the convergence assumption \eqref{conv_ass} holds, then $\chi$ moves by forced mean curvature according to 
	Definition \ref{def_motion_by_mean_curvature} with equation \eqref{H=v} replaced by
	\begin{align}\label{H=v+f}
	      \sum_{1\leq i,j\leq P} \sigma_{ij} \int_0^T \int  \left(\nabla\cdot\xi - \nu_i \cdot \nabla \xi \,\nu_i
	      -  V_i\, \xi \cdot \nu_i  \right) 
	       & \frac{1}{2}\left(|\nabla \chi_i|  + |\nabla \chi_j| - |\nabla (\chi_i + \chi_j)|\right) dt	\\
	      & \qquad \qquad = \sum_{1\leq i\leq P}\int_0^T \int  f \cdot \alpha_i (\xi\cdot \nabla) \,\chi_i  dt. \notag
	\end{align}
      \end{thm}
      
       Since we allow $f$ to be only of class $W^{1,2}$, the right-hand side of \eqref{H=v+f} has to be interpreted in the following distributional sense
      \[
       \int_0^T \int \left(f \cdot \alpha_i\right) (\xi\cdot \nabla) \,\chi_i dt
       = -  \int_0^T \int   \left(\nabla \cdot \xi\right) \left( f\cdot \alpha_i\right) \chi_i +\xi \cdot \nabla \left(f \cdot \alpha_i \right) \chi_i \, dx\,dt.
      \]
      
      In the volume preserving case we only deal with the scalar equation.
      
      \begin{thm}\label{thm VPAC2VPMCF}
	Let $N=1$.
       	Let $W$ satisfy \eqref{growth}, \eqref{growth_derivative} and \eqref{decomposition_convex} with zeros at $0$ and $1$, i.e.\ we have $P=2$.
       	Let $T<\infty$ be an arbitrary finite time horizon.
	Given a sequence of initial data $u_\eps^{0} \colon \torus \to \R$ approximating a characteristic function $\chi^0$, in the sense that
	\[
	  u^{0}_\eps \to \chi^0 \quad \text{a.e.\ and}\quad   E_0:= E(\chi^0) = \lim_{\eps\downarrow 0} E_\eps(u^{0}_\eps)<\infty,
	\]
	there exists a subsequence $\eps\downarrow 0$ such that
	the solutions $u_\eps$ of 
	\begin{equation}
	  \begin{cases}\label{allen cahn vol}
	    \partial_t u_\eps  = \Delta u_\eps - \frac{1}{\eps^2} W'(u_\eps) + \frac{1}{\eps}\lambda_\eps & \text{ in } [0,T]\times \torus \\
	    u_\eps = u_\eps^0 & \text{ on } {0}\times\torus
	  \end{cases}
	\end{equation}
	with 
	\begin{equation}\label{def lambda}
	\lambda_\eps := - \fint \eps \Delta u_\eps - \frac{1}{\eps} W'(u_\eps) dx = \fint \frac{1}{\eps} W'(u_\eps) dx 
	\end{equation}
	converge to a time-dependent characteristic function $\chi \in C([0,T];L^2(\torus;\{0,1\}^P))$ with
	\[\int \chi(t,x) dx \equiv \int \chi^0(x)dx.\]
	Furthermore, we have
	\[\sup_{\eps>0} \int_0^T  \lambda_\eps^2\, dt < \infty\]
	and there is a limit $\lambda_\eps \rightharpoonup \lambda$ in $L^2(0,T)$.
	If the convergence assumption \eqref{conv_ass} holds, then $\chi$ moves by volume preserving mean curvature according to 
	Definition \ref{def_motion_by_mean_curvature} with equation \eqref{H=v} replaced by
	\begin{align}\label{H=v+lambda}
	       \int_0^T \int \left(\nabla\cdot\xi - \nu \cdot \nabla \xi \,\nu
	      -  V\, \xi \cdot \nu  \right) 
	      |\nabla \chi| dt = - \int_0^T \lambda \int (\nabla \cdot \xi) \chi \,dx\, dt.
	\end{align}
      \end{thm}

      Throughout the paper we will make use of the following notations:
      The symbol $\partial_t$ denotes the time-derivative, $\nabla$ the spatial gradient of a function defined on real space $\R^d\ni x$ , 
      $\partial_u W(u)$ denotes the gradient of $W$ at a point  $u \in\R^N$ in state space.
      For the functions $\phi_i$ we will abuse the notation $\partial_u$ in the sense given by the generalized chain rule below, see Lemma \ref{chain rule}.
      We will write $A \lesssim B$ if there exists a generic constant $C<\infty$ depending only on $d,\,N, \Lambda$ and $W$ such that
      $A\leq C\,B$. 

    \section{Compactness}\label{sec:comp}
      \subsection{Results}
	Before we turn to the actual compactness results, we specify the setting for the Allen-Cahn Equation and make sure that solutions actually exist.
	
	Although solutions to the Allen-Cahn Equation \eqref{allen cahn} are smooth, we choose the weak setting for the following reasons:
	\begin{enumerate}
	 \item The parabolic character of both the Allen-Cahn Equation and mean-curvature flow is much more explicit.
	 \item It is the natural setting when including forces, which we will do later on in Section \ref{sec:forces,volume}.
	 \item Once one accepts the function spaces involved, the necessary compactness properties for forced equations and equations with a volume constraint and how to deal with initial conditions becomes very natural.
	\end{enumerate}
	We will essentially view solutions as maps of $[0,T]$ into some function space,
	so that we will need to deal with Banach space-valued $L^p$ and Sobolev spaces.
	However, the material covered in Chapter 5.9 of \cite{evansbigbook} is perfectly sufficient for our purposes.
	
	\begin{defn}
	 We say that a function 
	 $u_\eps\in C([0,T];L^2(\torus;\R^N))$ is a weak solution of the Allen-Cahn Equation \eqref{allen cahn} for $\eps>0$ with initial data $u_\eps^0\in L^2(\torus,\R^N)$ if
	 \begin{enumerate}
	  \item the energy stays bounded: \[\sup_{0\leq t\leq T} E_\eps(u_\eps(t))<\infty,\]
	  \item its weak time derivative satisfies
	    \[\partial_t u_\eps \in L^2([0,T]\times \torus),\]
	  \item for \ae $t \in [0,T]$ and $\xi \in L^p(0,T;\R^N) \cap W^{1,2}(0,T;\R^N)$ we have
	    \[ \int \partial_t u_\eps (t) \cdot \xi+ \nabla u_\eps(t) \colon \nabla \xi  + \frac{1}{\eps^2} \partial_u W(u_\eps(t))\cdot \xi\, dx = 0,\]
	  \item the initial conditions are achieved:
	    \[u(0) = u^0.\]
	 \end{enumerate}
	
	\end{defn}

	\begin{rem}
	 Note that due to the growth condition \eqref{growth_derivative} of $\partial_u W$ we know that
	 \[|\partial_u W(u)|^\frac{p}{p-1} \lesssim |u|^{(p-1)\frac{p}{p-1}} = |u|^p.\]
	 Combining this with boundedness of the energy and the growth condition \eqref{growth} of $W$ we get $\partial_u W(u(t)) \in L^\frac{p}{p-1} = L^{p'}$ for almost all times.
	 
	 Also note that boundedness of the energy and the bound on the time derivative are sufficient to have $u\in C^\frac{1}{2}([0,T];L^2(\torus))$, up to a 
	 set of measure zero in time, by the embedding
	 \[W^{1,2}([0,T];L^2(\torus)) \hookrightarrow C^\frac{1}{2}([0,T];L^2(\torus)).\]
	 See \eqref{sobolev_into_hoelder} for a short proof of a similar statement.
	\end{rem}
	
	We first take a brief moment to mention the (not very surprising) fact that the Allen-Cahn Equation \eqref{allen cahn} in fact have global solutions.
	For the convenience of the reader we later give a proof which relies on De Giorgi's minimizing movements and thus carries over to related settings.
	We point out that the long-time existence critically depends on the gradient flow structure, as solutions to the reaction-diffusion equation
	\[\partial_t u - \Delta u = u^2\]
	generically blow up in finite time.

	\begin{lem}\label{existence}
	  Let $u_\eps^0: \torus \to \R^N$ be such that $E_\eps(u_\eps^0) < \infty$.
	  Then there exists a weak solution $u: [0,T]\times \torus \to \R^N$ to the Allen-Cahn Equation $\eqref{allen cahn}$ with initial data $u^0$.
	  Furthermore, the solution satisfies the following energy dissipation identity
	  \begin{equation}\label{energy-dissipation equality}
	    E_\eps(u_\eps(T)) + \int_0^T \int \eps \left|\partial_t u_\eps \right|^2 dx\,dt = E_\eps(u_\eps(0))
	  \end{equation}
	  and we have $\partial_i\partial_j u, \partial_u W(u) \in L^2([0,T]\times \torus)$ for all $1\leq i,j \leq d$.
	  In particular, we can test the Allen-Cahn equations \eqref{allen cahn} with $\nabla u$.
	\end{lem}
	
	\begin{rem}\label{rmk dE/dt formal computation}
	Here, the identity \eqref{energy-dissipation equality} plays the role of an a priori estimate, which makes the whole machinery work.
	It can be formally derived by differentiating the energy along the solution:
	\begin{align*}
	  \frac{d}{dt} E_\eps(u_\eps) 
	  = & \int \eps \nabla u_\eps : \nabla \partial_t u_\eps + \frac1\eps \partial_u W(u_\eps) \cdot \partial_t u_\eps \,dx\\
	  = & \int  \eps \left( - \Delta u_\eps + \frac1{\eps^2} \partial_u W(u_\eps) \right) \cdot \partial_t u_\eps \,dx\\
	  \overset{\eqref{allen cahn}}{=} & -\int \eps |\partial_t u_\eps |^2 \, dx.
	\end{align*}
	\end{rem}

	\begin{rem}
	  Note that by choosing $W \equiv 0$ in this calculation, we get a similar estimate for the heat equation.
	  The structure of this estimate (the energy is bounded in time, while the time-derivative is only $L^2$-integrable) naturally leads to the mixed spaces we consider here and is our main justification for working in the weak setting.
	  
	  We also remark that the heat equation admits many different interpretations as a gradient flow.
	  Here we chose to view it as an $L^2$-gradient flow w.r.t.\ to the energy
	  $\int |\nabla u|^2 dx$
	  in order to compare it to the Allen-Cahn Equation.
	  However, when proving existence results for the heat equations it is more beneficial to interpret it as an $H^{-1}$-gradient flow w.r.t.\ to the energy
	  $\int u^2 dx$
	  as this choice allows to accommodate more general forces.
%
%
%
	\end{rem}
	
	\begin{rem}\label{compactness_for_forces}
	  As the a priori estimate is a natural consequence of the gradient flow structure we expect to have similar estimates in the case of forced equations and volume constraints.
	  In order to later deal with these more general equations we point out that the proofs of the following statements (Proposition \ref{prop_comp}, Lemma \ref{Upgraded_compactness}, Proposition \ref{prop_normal_velocities} and Lemma \ref{lem impl conv ass}) only rely on the a priori estimate \eqref{energy-dissipation equality} and not on the Allen-Cahn Equation \eqref{allen cahn} itself.
	  To be more precise, they remain valid - with slightly different quantitative estimates - for functions $u_\eps \in C([0,T];L^2( \torus ; \R^N))$ satisfying the bound
	  \begin{equation}\label{generalized_energy_dissipation}
	    \sup_{\eps>0} \sup_{0\leq t\leq T } E_{\eps}(u_\eps(t)) + \int_0^T \eps |\partial_t u_\eps|^2 dt < \infty.
	  \end{equation}
	\end{rem}

	We now turn to the central question of compactness for the constructed solutions:
	\begin{itemize}
	 \item Proposition \ref{prop_comp} ensures that there exists a time-dependent limiting partition, whose motion we want to characterize later on.
	 \item Lemma \ref{Upgraded_compactness} upgrades the convergence of $u_\eps$ to $\sum_i \chi_i \alpha_i$ to strong $C\left([0,T];L^2(\torus)\right)$ convergence, in particular implying that the initial conditions are achieved.
	 \item Proposition \ref{prop_normal_velocities} states that the partition is regular enough in time to admit normal velocities.
	\end{itemize}
	
	The existence of a limiting partition is essentially contained in the classical $\Gamma$-convergence theorem by Baldo \cite{baldo1990minimal}.
	In particular, it is constructed by considering the limits of $\phi_i \circ u_\eps$ with
	\begin{equation}  \label{phi_i}
	  \phi_i(u) := d_W(u,\alpha_i),\text{ where $d_W$ was defined in \eqref{geodesic distance}.}
	\end{equation}
	The main difference is that we also want the partition to be well-behaved in time, which we will make sure by exploiting that the control of $\partial_t u_\eps$ and $\nabla  u_\eps$ is similar.
	
	\begin{prop}\label{prop_comp}
	  Given initial data $u_\eps^0 \to \sum_i \chi_i^0\alpha_i$ with 
	  \[
	    E_\eps(u_\eps^0) \to E(\chi^0)<\infty,
	  \]	
	  for any sequence there exists a subsequence $\eps\downarrow0$ such that the solutions $u_\eps$ of \eqref{allen cahn}
	  converge:
	  \begin{align}\label{u eps to chi}
	    u_\eps \to u \quad \ae \text{in } (0,T) \times \torus.
	  \end{align}
	  Here the limit is given by $u = \sum_i \chi_i \alpha_i$ with a partition $\chi\in BV( (0,T)\times \torus; \{0,1\}^P)$.
	  Furthermore we have
	  \[
	    \sup_{0\leq t \leq T} E(\chi) \leq E_0
	  \]
	  and the compositions $\phi_i\circ u_\eps$ are uniformly bounded in $BV((0,T)\times \torus)$ and converge:
	  \begin{equation}\label{phi eps to phi}
	    \phi_i\circ u_\eps \to \phi_i\circ u \quad \text{in } L^1([0,T]\times\torus).
	  \end{equation}
	\end{prop}
	
	In the following lemma, we record some properties of the functions $\phi \circ u_\eps$, such as the estimates going back to Modica and Mortola by which one deduces $BV$-compactness of these compositions.
	The main point is however that we will need more precise information about $\phi \circ u_\eps$ than for the previously known $\Gamma$-convergence results, where one only needs upper bounds for $|\nabla (\phi \circ u_\eps)|.$ 
	
	Because our proof works by multiplying the Allen-Cahn equation \eqref{allen cahn} with $\eps \xi \cdot \nabla u$, we will need to pass to the limit in non-linear quantities of $u_\eps$, such as $\int \eta \sqrt{2W(u_\eps)}\nabla u_\eps$.
	For scalar equations one can easily identify the limit by applying the chain rule to see that this non-linearity has the form $\nabla (\phi \circ u_\eps)$, where the primitive $\phi$ is given by $\phi(u):= \int_{\alpha_1}^u\sqrt{2W(\tilde u)}\, d\tilde u$.
	In the multi-phase case, unfortunately, the classical chain rule does not apply anymore:
	Because there could be multiple geodesics between $u$ and $\alpha_i$, the geodesic distances $\phi_i(u)$, playing the roles of ``primitives'', are only locally Lipschitz-continuous in general.
	
	Luckily, there is a chain rule for Lipschitz functions due to Ambrosio and Dal Maso \cite{ambrosio1990general}.
	The upshot is that given a Lipschitz function $f$ and a function $u$ there exists a bounded function $g(x,u)$, defined almost everywhere, such that
	\[D (f \circ u)(x) = g(x,u) D u(x)\]
	and the dependence of $g$ on $u$ is local in $x$, but not pointwise.
	See Theorem \ref{chain rule, general} in the proof of Lemma \ref{chain rule} for the precise formulation.
	
	The following lemma mainly serves to fix and justify our somewhat abusive notation of these differentials.
      
	\begin{lem}\label{chain rule}
	  Let $u \in C([0,T];L^2( \torus ; \R^N))$ with
	  \[\sup_{0\leq t\leq T} E_\eps(u) + \int_0^T \int \eps |\partial_t u|^2 dx\, dt < \infty\]
	  for some $\varepsilon>0$.
	  Then for all $1\leq i \leq P$ there exists a map 
	  \[\partial_u \phi_i(u) : [0,T]\times\torus \to \operatorname{Lin}(\R^N;\R)\]
	   such that the chain rule is valid with the pair $\partial_u \phi_i(u)$ and $(\partial_t, \nabla) u$:
	  For almost every $(t,x)\in [0,T]\times \torus$ we have 
	  \begin{align}\label{Dphi}
	    \nabla \left( \phi_i \circ u\right)  = \partial_u \phi_i(u) \nabla u\quad \text{and} \quad
	    \partial_t \left( \phi_i \circ u\right)   = \partial_u \phi_i(u)  \partial_t u.
	  \end{align}
	  Furthermore, we can control the modulus of $\partial_u \phi_i(u)$ almost everywhere in time and space:
	  \begin{align}
	    |\partial_u \phi_i(u)| \leq \sqrt{2W(u)}.\label{Lipschitz estimate}
	  \end{align}
	  Additionally, we have $\phi_i\circ u \in  L^\infty\left([0,T];W^{1,1}(\torus)\right) \cap W^{1,1}([0,T]\times \torus)$ with the estimates
	  \begin{align}
	  \sup_{0\leq t \leq T} \int |\phi_i\circ u| \, dx & \lesssim 1 + \sup_{0\leq t\leq T} \eps E_\eps(u), \label{L1_estimate_space}\\
	  \sup_{0\leq t \leq T} \int  |\nabla ( \phi_i\circ u)|\, dx & \lesssim \sup_{0\leq t\leq T} E_\eps(u), \label{BV_estimate_space}\\
	  \int_0^T \int  |\partial_t (\phi_i\circ u)| dx dt & \lesssim T \sup_{0\leq t\leq T} E_\eps(u) + \int_0^T \int \eps |\partial_t u|^2 dx\, dt. \label{BV_estimate_time}
	  \end{align}
	\end{lem}
	
	Next, we turn to the stronger compactness properties of $u_\eps$.
	In the case of the Allen-Cahn Equation without forces or constraints, it mainly serves to ensure that the initial data is achieved.
	When including forces or constraints we will also need it in the proof of the actual convergence.
	
	\begin{lem}\label{Upgraded_compactness}
	  We have $\phi_i \circ u_\eps \in W^{1,2}([0,T];L^1(\torus))$ with the estimate
	  \begin{align}
	    \left( \int_0^T \left(\int |\partial_t (\phi_i \circ u_\eps) | dx \right)^2 dt \right)^\frac{1}{2}  \lesssim E_\eps(u_\eps(0)).\label{L2L1-estimate_derivative}
	  \end{align}
	  Furthermore, the sequence $u_\eps$ is pre-compact in $C\left([0,T];L^2(\torus;\R^N)\right)$.
	  In particular, we get that $\chi$ achieves the initial data in $C([0,T];L^2(\torus))$.
%
	\end{lem}
	
	Note that the estimate \eqref{L2L1-estimate_derivative} and the embedding $W^{1,2}([0,T]) \hookrightarrow C^{\frac{1}{2}}([0,T])$,
	see \eqref{sobolev_into_hoelder} for a short proof for Banach space-valued functions, imply the well-known $\frac{1}{2}$-H\"older
	continuity of the volumes of the phases.
	
	The proof of this lemma makes the most detailed use of mixed spaces.
	Estimate \eqref{L2L1-estimate_derivative} is a time-localized version of the $BV$-compactness in time \eqref{BV-compactness_in_time}.
	Uniform convergence in time of $\phi_i\circ u_\eps$ then boils down to combining this estimate with the Arzel\`a-Ascoli theorem.
	However, passing this convergence to $u_\eps$ is a little delicate because we have no quantitative information about how quickly $\phi_i$ grows around $\alpha_i$.
	Consequently, we have to make do with $u_\eps$ only converging in measure uniformly in time.

	While the compactness statement, Proposition \ref{prop_comp}, did not rely on the convergence assumption \eqref{conv_ass} we will need to 
	assume it in the following, starting with the existence of the normal velocities.
      
	\begin{prop}\label{prop_normal_velocities}
	  In the situation of Proposition \ref{prop_comp}, given the convergence assumption \eqref{conv_ass}, for every $1\leq i \leq P$ the measure $\partial_t \chi_i$ is absolutely 
	  continuous w.r.t. $\left|\nabla\chi_i\right| dt$
	  and the density $V_i$ is square-integrable:
	  \begin{equation}\label{V squared}
	    \int_0^T \int V_i^2 \left|\nabla\chi_i\right| dt \lesssim E_0.
	  \end{equation}
	  Furthermore, equation \eqref{v=dtX} holds.
	\end{prop}
	
	While we previously localized the $BV$-compactness in time \eqref{BV-compactness_in_time}, for this statement we need to localize it in space.
	Unfortunately, the argument is somewhat delicate as one first proves $\partial_t \chi_i \ll E(\bullet, u) dt$ and then is forced to prove that $\partial_t \chi_i$ is singular to the ``wrong'' parts of the energy.
	
	Finally, the following lemma shows that -- up to a further subsequence -- the convergence assumption can be refined to pointwise $\ae$ in time and can be localized by a smooth test function in space.
	We furthermore argue that our convergence assumption assures equipartition of energy as $\eps\downarrow0$.
      
	\begin{lem}\label{lem impl conv ass}
	  Given $u_\eps \to u$ and the convergence assumption \eqref{conv_ass}, by passing to a further subsequence if necessary, we have
	  \begin{equation}\label{conv ass pw in t}
	    \lim_{\eps\downarrow0} E_\eps(u_\eps) = E(u)\quad \text{for }\ae \; 0\leq t \leq T
	  \end{equation}
	  and for any smooth test function $\zeta\in C^\infty(\torus)$ we have
	  \begin{equation}\label{equipartition}
	    E(\zeta, u)
	    = \lim_{\eps\downarrow0} E_\eps(\zeta,u)
	    = \lim_{\eps\downarrow0} \int \zeta  \eps \left| \nabla u_\eps\right|^2 dx
	    = \lim_{\eps\downarrow0} \int \zeta \frac2\eps W(u_\eps)\, dx
	    = \lim_{\eps\downarrow0}  \int \zeta\sqrt{2W(u_\eps)} \left| \nabla u_\eps \right| dx
	  \end{equation}
	  for $\ae \;0 \leq t \leq T.$
	\end{lem}
      
      A key ingredient for this lemma to work  was already observed by Baldo, see Proposition 2.2 in \cite{baldo1990minimal}: the optimal partition energy \eqref{E} can be written as a
      (measure-theoretic) supremum using the ``primitives'' $\phi_i$ defined in \eqref{phi_i}.
      We will use this fact in the following form:
      Given $\varepsilon>0$ there exists a scale $r>0$ such that 
      \begin{align}\label{baldo E is supremum}
	\sum_{B\in \B_r} \left\{ E(\eta_B, u ) -   \max_{ 1\leq i \leq \numphases} \int \eta_{B} \left| \nabla \left( \phi_i\circ u\right) \right| \right\} \leq \varepsilon,
      \end{align}
      where $\eta_B$ is a cutoff for $B$ in the ball $2B$ with the same center but with the double radius and the covering $\B_r$ is given by
      \begin{align}\label{covering Br}
		\B_r := \left\{ B_r(i) \colon i\in \L_r \right\}
      \end{align}
      of $\torus$, where $\L_r = \torus \cap \frac r{\sqrt{d}} \Z^d $ is a regular grid of midpoints on $\torus$.
      Let us note that each summand in \eqref{baldo E is supremum} is non-negative:
      \[
       0\leq E( u,\eta_B ) -   \max_{ 1\leq i \leq \numphases} \int \eta_{B} \left| \nabla \left( \phi_i\circ u\right)\right|.
      \]

      This is the same covering as in Definition 5.1 in \cite{laux2015convergence}.
      A nice feature is that by construction, for each $n\geq 1$ and each $r>0$, the covering
	      \begin{align}\label{finiteoverlap}
	      \left\{ B_{nr}(i) \colon i \in \L_r\right\} \quad \text{is locally finite,}
	      \end{align}
	      in  the sense that for each point in $\torus$, the number of balls containing
	      this point is bounded by a constant $c(d,n)$ which is independent of $r$.
      
      We will later also apply this covering to exploit that $BV$-partitions generically only have a single, essentially flat interface on small scales, where flatness is measured by the variation of the normal, i.e.\ the tilt-excess mentioned earlier.
      This is ensured by the following fact, which is a direct consequence of \cite[Lemma 5.2 and Lemma 5.3]{laux2015convergence}.
	      \begin{lem}
	      \label{lem_approximate_normal_L2}
	      For every $\kappa>0$ and $\chi: \torus \to \{0,1\}^P$ with $\sum_{1\leq i\leq P} \chi_i = 1$,
	      there exists an $r_0>0$ such that for all $r\leq r_0$ the following holds :
	      There exist unit vectors 
	      $\nu_B\in \sph^{d-1}$ for all $B \in \B_r$ such that
	      \begin{align}\label{approximate_normal_L2}
	      \sum_{B\in\B_r}
	      \min_{i\neq j}
		\Bigg\{ 
		 \int \eta_B\left| \nu_{i} - \nu_B \right|^2 \left| \nabla \chi_{i}\right|
		+ \int\eta_B \left| \nu_{j} + \nu_B \right|^2 \left| \nabla \chi_{j}\right| +\sum_{k\notin \{i,j\}}\int\eta_B \left| \nabla \chi_k\right|
	      \Bigg\}
	      \lesssim \kappa E(\chi).
	      \end{align}
	      \end{lem}

      \subsection{Proofs}
      \begin{proof}[Proof of Lemma \ref{existence}]
	\textit{Step 1: Existence via minimizing movements.}\\
	Since $\eps$ is fixed, we may set $\eps=1$ and denote the Ginzburg-Landau energy by $E$.
	For a fixed time-step size $h>0$ and $n \in \mathbb{N}$ we inductively set
	\[
	 u^n = \arg \min_u \left\{ E(u) + \frac1{2h} \int \left| u-u^{n-1}\right|^2 dx \right\}.
	\]
	The existence of minimizers $u^n$ follows from the direct method since both $E$ and the metric term $\frac1{2h} \int \left| u-u^{n-1}\right|^2 dx$ are lower semi-continuous w.r.t.\ weak convergence in $H^1$.
	Note however that some care needs to be taken in the term $\int W(u)$, as $W$ is non-convex and and the Rellich compactness theorem is applicable in the case $p \geq 2^\ast = \frac{2d}{d-2}$.
	Using the decomposition \eqref{decomposition_convex} one can still deduce lower semi-continuity in these cases as $W_{conv}$ is convex and the non-convexity in $W_{pert}$ can be treated using Rellich's compactness theorem.
	
	We interpolate these functions in a piecewise constant way: $u^h(t) := u^n$ for $t\in [nh, (n+1)h)$.
	By comparing $u^n$ to its predecessor $u^{n-1}$ we obtain for any $T=Nh$ the a priori estimate
	\[
	  E(u^h(T)) + \frac12 \int_0^{T} \int \left| \partial_t^h u^h \right|^2 dx\,dt \leq E(u^0),
	\]
	where $\partial_t^h u (t)= \frac{u(t+h) - u(t)}{h}$ denotes the discrete time-derivative of a function $u$.
	By the estimate
	\[||u^h(t+nh) - u^h(t)||_{L^2} \leq \int_t^{t+nh}\left|\left|\partial_t^h u(s) \right|\right|_{L^2}ds \leq E(u^0)^\frac{1}{2} (nh)^\frac{1}{2}\]
	for $t+(n+1)h\leq T$ one can deduce compactness:
	There exists a sequence $h\downarrow 0$ and a limiting function $u\in C^\frac{1}{2}([0,T]; L^2(\torus; \R^N))$ such that
	\[
	  u^h \to u  \text{ in } L^2(\torus) \text{ for all times } 0\leq t \leq T
	\]
	and furthermore
	\[
	\sup_t \int \left| \nabla u\right|^2 dx, \quad \int_0^T \int \left| \partial_t u \right|^2 dx\,dt <\infty.
	\]
	
%
	
	We want to pass to the limit $h\downarrow0$ in the the Euler-Lagrange equation
	\[
	\int_0^T \int \partial_t^{-h} \xi \cdot u^h  +\Delta \xi \cdot  u^h \,dx\,dt = \int_0^T \int \xi \cdot \partial_u W(u^h) \,dx\,dt
	\]
	for all test vector fields $\xi \in C_0^\infty((0,T)\times \torus, \R^N)$.
	By the pointwise convergence we have
	\[ 
	  \partial_u W(u^h) \to \partial_u W(u) \quad a.e.
	\]
	By the polynomial growth conditions \eqref{growth} and \eqref{growth_derivative} of $W$ we have
	\[|\partial_u W(u^h)|^{\frac{p}{p-1}} \lesssim |u^h|^p,\]
	which implies that $\sup_h ||\partial_u W(u^h)||_{L^{\frac{p}{p-1}}} < \infty$.
	Thus the sequence $|\partial_u W(u^h)|$ is equi-integrable, which implies that $\partial_u W(u^h) \to \partial_u W(u)$ in $L^1$.

	\textit{Step 2: We have $\partial_i\partial_j u , \partial_u W(u) \in L^2$.}\\
	We provide a formal argument which can easily be turned into a rigorous proof by considering discrete difference quotients instead of their limits.
	Differentiating the equation in the $i^{\text{th}}$ coordinate direction for $1\leq i\leq d$ gives
	\[\partial_t \partial_i u - \Delta \partial_i u = - \partial_u^2W(u)\partial_i u.\]
	By multiplying the equation with $\partial_i u$ and integrating we find
	\[ \frac12 \int |\partial_i u(T)|^2 dx + \int_0^T \int |\partial_i \nabla u|^2 dx\,dt = \frac12 \int|\partial_i u(0)|^2 dx - \int_0^T \int \partial_i u \cdot \partial^2_u W(u) \partial_i u \,dx\,dt.\]
	The second right-hand side term has two contributions, one from $W_{conv}$ and one from $W_{pert}$, see \eqref{decomposition_convex}.
	The contribution due to $W_{conv}$ is negative by convexity.
	The contribution coming from $W_{pert}$ is controlled by
	\[\int_0^T \int |\partial_i u|^2 dx\,dt\]
	because $W_{pert}$ has bounded second derivative.
	Thus we get $\partial_i\partial_j u \in L^2([0,T]\times \torus)$.
	As $\partial_t u$ is in the same space, a quick look at the PDE \eqref{allen cahn} reveals that $\partial_u W(u)$ is as well.
	
	Finally, the equality \eqref{energy-dissipation equality} follows from integrating the outcome of the computation in Remark \ref{rmk dE/dt formal computation} from $0$ to $T$.
	
	\end{proof}

      \begin{proof}[Proof of Proposition \ref{prop_comp}]
%
%
	Plugging the a priori estimate \eqref{energy-dissipation equality} into the estimates \eqref{L1_estimate_space}, \eqref{BV_estimate_space} and \eqref{BV_estimate_time} of Lemma \ref{chain rule} we see that
	\[\sup_{\eps} \int_0^T   \int | \phi\circ u_\eps| + |\nabla (\phi_i \circ u_\eps)| + |\partial_t (\phi_i \circ u_\eps)|\, dx \,dt < \infty.\]
	By the Rellich compactness theorem, we thus find a subsequence $\eps\downarrow 0$ and a function $v\colon (0,T)\times \torus \to \R$ such that
	\begin{align}\label{phi to v}
	\phi_i(u_\eps) \to v\quad \text{in } L^1([0,T]\times \torus).
	\end{align}
      \textit{Step 1: The limit $v$ takes the form $\sum_j \phi_i(\alpha_j)\chi_j$ and the functions $u_\eps$ converge to $\sum_j \chi_j\alpha_j$.}\\
	The convergence of $u_\eps$ to $\sum_j \chi_j\alpha_j$ is a part of the classical $\Gamma$-limit result \cite{baldo1990minimal}.
	However, we take this opportunity to provide a clarification of the previously known argument.
	
	After passing to another subsequence we can assume that the sequence $u_\eps$ generates a Young measure $p_{t,x}$.
	We note that
	\[
	  \int_0^T \int W(u_\eps) \,dx\,dt \to 0
	\]
	implies that $u_\eps$ tends to the zeros of $W$ in measure: For any $\delta>0$ we have
	\[
	  \left| \left\{ (x,t) \colon \operatorname{dist}(u_\eps, \{\alpha_1 , \dots, \alpha_\numphases \}) \geq \delta\right\} \right| \to 0 \quad \text{as } \eps\to 0.
	\]
	Hence also the Young measure concentrates and we get
	\[p_{t,x} = \sum_{1\leq j \leq P} p_{t,x}(\alpha_j) \delta_{\alpha_j}.\]
	From this estimate we also get that no mass escapes to infinity, i.e.\ $\sum_{1\leq j\leq P} p_{t,x}(\alpha_j) =1$.
	
	By \eqref{phi to v} for all $f \in C_c(\R)$ also $f\circ \phi_i (u_\eps)$ is strongly compact in $L^1$.
	Therefore Young measure theory gives us the following (a.e.) identities:
	\begin{align}
	v & = \lim \phi_i(u_\eps)  = \sum_{1\leq j \leq P} \phi_i(\alpha_j) p_{t,x}(\alpha_j), \label{Young_measure_phi}\\
	f\left(\sum_{1\leq j \leq P} \phi_i(\alpha_j) p_{t,x}(\alpha_j)\right) & = \lim f(\phi_i(u_\eps))  = \sum_{1\leq j \leq P} f(\phi_i(\alpha_j)) p_{t,x}(\alpha_j).\notag
	\end{align}
	If we take $f$ to be uniformly convex on the interval $[0, \max_{1\leq j \leq P} \phi_i(\alpha_j)]$ we see from the equality statement in Jensen's inequality that
	\[\phi_i(\alpha_j) \equiv \sum_{1\leq k \leq P} \phi_i(\alpha_{k}) p_{t,x}(\alpha_{k}) \quad p_{t,x} \text{-almost surely.}\]
	For $(t,x)\in [0,T]\times \torus$ (up to a set of measure zero) let $\alpha_j$ be such that $p_{t,x}(\alpha_j)>0$.
	We then get that
	\[0 = \phi_j (\alpha_j) = \sum_{1\leq k \leq P} \phi_j(\alpha_{k}) p_{t,x}(\alpha_{k}) = \sum_{k \neq j} \phi_j(\alpha_{k}) p_{t,x}(\alpha_{k}).\]
	Since $\phi_j(\alpha_k)>0$  for $k\neq j$ we have $p_{t,x}(\alpha_k)=0$.
	Thus we get $p_{t,x}(\alpha_j)=1$.
	Setting $\chi_i(t,x) := p_{t,x}(\alpha_i)$ and inserting this definition into equation \eqref{Young_measure_phi} proves the decomposition of $v$.
	
	In order to get pointwise \ae convergence of $u_\eps$, note that since the Young measures concentrate, we get that $u_\eps \to \sum_j \chi_j \alpha_j$ in measure.
	By passing to a subsequence, we can upgrade this to pointwise almost everywhere convergence.
	
      \textit{Step 2: $\chi_i \in BV$.}\\
	A similar claim is proven to be true in Prop. 2.2 in \cite{baldo1990minimal}.
	For the convenience of the reader and later refinement we reproduce the proof.
	
	Applying the Fleming-Rishel coarea formula in space and time we see for each $1\leq i\leq P$ that
	\begin{align*}
	  ||(\partial_t, \nabla) \phi_i \circ u_\eps ||_{TV}
	  = & \int_{-\infty}^\infty \mathcal{H}^{d}\left(\partial_\ast \{(t,x):\phi_i\circ u_\eps (t,x) \leq s\}\right) ds \\
	  \geq & \int_0^{d_i} \mathcal{H}^{d}\left(\partial_\ast \{(t,x):\phi_i\circ u_\eps (t,x) \leq s\}\right) ds\\
	  = & d_i ||(\partial_t,\nabla) \chi_i||_{TV},
	\end{align*}
	where we define $d_i:= \min_{1\leq j \leq P, i\neq j} d_W(\alpha_i, \alpha_j)$.
	Thus $\chi_i \in BV([0,T]\times\torus)$.
	
	For the statement $||E(\chi)||_{L^\infty([0,T])} \leq E_0$ we refer the reader to the proof of the $\Gamma - \liminf$ inequality in \cite{baldo1990minimal}
	and the energy-dissipation equality \eqref{energy-dissipation equality}.
	
	Finally, recalling Remark \ref{compactness_for_forces} we notice that the Allen-Cahn Equation only played into the argument via the energy-dissipation estimate \eqref{energy-dissipation equality}.
      \end{proof}

      \begin{proof}[Proof of Lemma \ref{chain rule}]
	\textit{Step 1: The chain rule holds if $u$ additionally is bounded in space and time.}\\
	In this case $\phi_i$ is in fact Lipschitz continuous on the image of $u$.
	By the following Theorem \ref{chain rule, general} due to Ambrosio and Dal Maso we know that the chain rule is valid for the pair $D(\phi_i|_{T_{t,x}})$ and $(\partial_t,\nabla) u$, where $\dot T_{t,x}:=\operatorname{span}\left(\{\partial_1 u, \ldots, \partial_d u, \partial_t u\}\right)$
	and $T_{t,x} := u (t,x) + \dot T_{t,x}$:
	
	\begin{thm}[Ambrosio, Dal Maso \cite{ambrosio1990general}; Corollary 3.2]\label{chain rule, general}
	  Let $\Omega \subset \R^d$ be an open set.
	  Let $p \in [1,\infty]$, $u \in W^{1,p}(\Omega;\R^N)$, and let $f:\R^N \to \R^k$ be a Lipschitz continuous function such that $f(0)=0$.
	  Then $v := f\circ u \in W^{1,p}(\Omega;\R^k)$.
	  Furthermore, for almost every $x\in \Omega$ the restriction of the function $f$ to the affine space
	  \[T_x^u := \left\{y\in \R^n: y= u(x)+ (z\cdot D)u \text{ for some } z \in \R^d\right\}\]
	  is differentiable at $u(x)$ and
	  \[D v = D (f|_{T_x^u})(u)D u \quad \ae \text{ in } \Omega.\]
	\end{thm}

	Let $\Pi(t,x)$ be the orthogonal projection in $\R^N$ onto the subspace $\dot T_{t,x}$ and let
	\[\partial_u \phi_i(u)(t,x) v:= D \big(\phi_i|_{T_{t,x}}\big)(u(t,x)) \Pi(t,x) v.\]
	Due the obvious fact that $\Pi(t,x) \nabla u(t,x) = \nabla u(t,x)$ the chain rule still holds for $\partial_u \phi_i(u)$ and $(\partial_t, \nabla) u$.
	Let $(t,x)$ be a point such that $\phi_i|_{T_{t,x}}$ is differentiable in $u:=u(t,x)$, let $v \in \dot T_{t,x}$ and $h>0$.
	Using the triangle inequality of $d$ and comparing the length of geodesics to straight lines we get
	\[|\phi_i(u+hv)-\phi_i(u)|\leq d_W(u+hv,u) \leq \int_0^1 \sqrt{2W(u+thv)}h|v|\, dt.\]
	Continuity of $W$ implies that we can pass to the limit $h \to 0$ to get
	\[\left|D \phi_i|_{T_{t,x}}(u)v\right| \leq \sqrt{2W(u)}|v|,\]
	which for all vectors of the form $v = \Pi(t,x) \tilde v$ for some $\tilde v \in \R^N$ gives
	\[\left|\partial_u \phi_i(u)\right| \leq \sqrt{2W(u)}.\]
	
	\textit{Step 2: The lemma holds for general functions $u$ with bounded energy and controlled dissipation.}\\
	The idea is to approximate $u$ with bounded functions.
	Let $M>0$ and let $u_{M,j} := \sign(u_j)  \left( M \wedge |u_j| \right) $ for all $1\leq j\leq N$ be the componentwise truncation of $u$.
	We then know that $u_M \to u$ pointwise almost everywhere, which implies $\phi_i(u_M) \to \phi_i(u)$ pointwise almost everywhere.
	Next, we will strengthen this to $L^1$ convergence by finding an integrable dominating function.
	
	By the triangle inequality for $d$ we get for all $v \in \R^N$ that
	\begin{align}\label{triangle inequality for geodesic distance}
	  \phi_i(v) \leq d_W(\alpha_i,0) + d_W(0,v),
	\end{align}
	so that it is sufficient to consider $d_W(0,v)$.
	By the growth condition \eqref{growth} on $W$ we see
	\begin{align}\label{growth_of_phi}
	 d_W(0,v) \leq \int_0^1 \sqrt{2W(sv)}|v| ds \lesssim |v| + |v|^{\frac{p}{2}+1} \lesssim 1 + |v|^p
	\end{align}                                                                                                    
	for all $v\in \R^N$.
	Thus we have
	\[\phi_i\left(u_M\right) \lesssim 1 + \left|u_M\right|^p \leq 1 + \left|u\right|^p\]
	and we only need to prove $L^p$-boundedness of $u$.
	This is a straightforward consequence of the coercivity assumption \eqref{growth} and boundedness of the energy, as for almost all times $0\leq t\leq T$ we have
	\begin{equation}\label{u goes to zeros}
	\sup_{0\leq t\leq T} \int |u|^p dx  \overset{\eqref{growth}}{\lesssim} \sup_{0\leq t\leq T} \int 1+W(u) dx \lesssim 1 + \sup_{0\leq t\leq T} \eps E_\eps(u).
	\end{equation}
	
	Thus we can apply Lebesgue's dominated convergence theorem to see that $\phi_i (u_M) \to \phi_i (u)$ in $L^1$.
	Consequently, we have that
	\[(\partial_t, \nabla)(\phi_i \circ u_M) \to (\partial_t, \nabla)(\phi_i \circ u) \]
	as distributions.
	
	Note that estimates \eqref{triangle inequality for geodesic distance}, \eqref{growth_of_phi} and \eqref{u goes to zeros} imply the $L^1$ estimate \eqref{L1_estimate_space} we claimed to hold in the statement
	of the lemma.
	
	By an elementary property of weakly differentiable functions we have that
	\[(\partial_t,\nabla) u_{M,j} = (\partial_t,\nabla) u_{j} \text{ \ae on } \left\{u_{M,j} = u_{j}\right\}.\]
	Since the sets $\left\{u_M = u\right\}$ are non-decreasing in $M$ we see that $\left|\left\{u_M \neq u , (\partial_t, \nabla)u_M \neq (\partial_t,\nabla)u\right\}\right| \to 0$.
	Because the definition of $\partial_u \phi_i$ only depends on the values of the pre-composed function and its derivatives, we see that $\partial_u \phi_i(u_M)$ eventually becomes stationary almost everywhere.
	We denote the limit by $\partial_u \phi_i(u)$.
	Furthermore, we still have
	\[|\partial_u \phi_i(u)| \leq \sqrt{2W(u)} \quad \text{a.e.},\]
	which proves \eqref{Lipschitz estimate}.
	Finally, to check the chain rule all remains to be seen is that \[\partial_u \phi(u_M) (\partial_t,\nabla)u_M \to \partial_u \phi(u) (\partial_t, \nabla) u\]
	in $L^1$.
	This follows by dominated convergence from the above pointwise convergences and the following widely known application of Young's inequality
	\[
	  \left|\partial_u \phi(u_M) \nabla u_M \right|  \leq \sqrt{2W(u_M)}\left|\nabla u_M\right|  \lesssim  \sqrt{2W(u)}| \nabla u|  
	  \leq  \frac{\eps}{2} |\nabla u|^2 +  \frac{1}{\eps} W(u)
	\]
	for the spacial gradient and, similarly,
	\begin{equation}\label{BV-compactness_in_time}
	  \left|\partial_u \phi(u_M) \partial_t u_M \right|  \leq \frac{\eps}{2} |\partial_t u|^2 +  \frac{1}{\eps} W(u) 
	\end{equation}
	as the right hand side is integrable in space and time by assumption.
	Note that both inequalities also imply
	\begin{align*}
	 \sup_{0\leq t \leq T} \int |\nabla \phi_i \circ u|  dx & \lesssim \sup_{0\leq t\leq T} E_\eps(u), \\
	 \int_0^T \int  |\partial_t \phi_i \circ u| dx dt & \lesssim T \sup_{0\leq t\leq T} E_\eps(u) + \int_0^T  \int \eps |\partial_t u|^2 dx \, dt,
	\end{align*}
	which provides the bounds \eqref{BV_estimate_space} and \eqref{BV_estimate_time}.
      \end{proof}

      \begin{proof}[Proof of Lemma \ref{Upgraded_compactness}]
      \textit{Step 1: We have $\phi_i \circ u_\eps \in W^{1,2}([0,T];L^1(\torus))$.}\\
	The fact that $\phi_i\circ u_\eps \in L^2([0,T];L^1(\torus))$ is an immediate consequence of estimate \eqref{L1_estimate_space} of Lemma \ref{chain rule}.
	For the estimate on the derivative we localize the previous estimate for $\partial_t (\phi_i \circ u_\eps)$ in time.
	Let $\zeta \in L^2([0,T])$ be non-negative.
	Using the chain rule \eqref{Dphi}, the Lipschitz estimate \eqref{Lipschitz estimate} and the Cauchy-Schwarz inequality in the spacial integral, we obtain
	\begin{align*}
	  \int_0^T \zeta \int | \partial_t (\phi_i\circ u_\eps)| dx\, dt & \leq \int_0^T \zeta \int \sqrt{2W(u_\eps)}|\partial_t u_\eps| dx \, dt \\
	  & \leq \int_0^T \zeta \left(2 \int \frac{1}{\eps} W(u_\eps) dx \right)^\frac{1}{2} \left(\int \eps |\partial_t u_\eps|^2 dx \right)^\frac{1}{2} dt.
	\end{align*}
	Applying the energy dissipation estimate \eqref{energy-dissipation equality} and the Cauchy-Schwarz inequality in time we arrive at
	\[ \int_0^T \zeta \int | \partial_t (\phi_i\circ u_\eps)| dx\, dt \lesssim E_\eps(u_\eps(0))\left(\int_0^T \zeta^2 dt\right)^\frac{1}{2}.\]
	Optimizing in $\zeta$ with $||\zeta||_{L^2} =1$ gives the $L^2_tL^1_x$-estimate \eqref{L2L1-estimate_derivative}.
	
      \textit{Step 2: The sequence $\phi_i \circ u_\eps$ is pre-compact in $L^\infty([0,T];L^1(\torus))$.}\\
	Due to a version of the Fundamental Theorem of Calculus for the Bochner integral, cf.\ Chapter 5.9, Theorem 2 in \cite{evansbigbook},
	  we know for almost every $s,r \in [0,T]$ with $s\leq r$ that
	\[\phi_i \circ u_\eps(r) -\phi_i \circ u_\eps(s) = \int_s^r \partial_t (\phi_i\circ u_\eps)(t) dt.\]
	Consequently, the Cauchy-Schwarz inequality gives
	\begin{equation}\label{sobolev_into_hoelder}
	 \int |\phi_i \circ u_\eps(r) -\phi_i \circ u_\eps(s) | dx \leq \int_s^r \int |\partial_t (\phi_i \circ u_\eps)(t)| dx\, dt \lesssim (r-s)^\frac{1}{2}\left(E_\eps(u_\eps^0)\right)^\frac{1}{2}.
	\end{equation}
	By estimate \eqref{BV_estimate_space} we also know that
	\[\esssup_{0\leq t \leq T} \int | \nabla (\phi_i \circ u_\eps)| dx \lesssim 1+ E_\eps(u_\eps^0).\]
	Since $\sup_\eps E_\eps(u_\eps^0) < \infty$ we consequently know that (a modification of) $\phi \circ u_\eps$ is equi-continuous in $C([0,T];L^1(\torus))$.
	Additionally, lower semi-continuity of the $BV$-norm and the compact Sobolev embedding of $W^{1,1}$ into $L^1$ implies that for all times $t \in [0,T]$ the maps $\phi_i\circ u_\eps(t)$ are pre-compact in $L^1(\torus)$.
	The Arzel\`a-Ascoli theorem then gives the claim.
	
      \textit{Step 3: The sequence $u_\eps$ converges to $\sum_i \chi_i \alpha_i$ in measure uniformly in time.}\\
	By $d_W(\alpha_i,\alpha_i)=0$ for all $1\leq i \leq P$ and Step 2 we get
	\[\esssup_{0\leq t \leq T} \sum_{i=1}^P \int d_W(\alpha_i, u_\eps(t,x)) \chi_i dx \leq \esssup_{0\leq t \leq T} \sum_i \int |d_W(\alpha_i, u_\eps(t,x)) - d_W(\alpha_i, u(t,x))| dx \to 0.\]
	For every $\delta >0$ and $1\leq i \leq P$ we have by continuity of the map $v \to d_W(\alpha_i, v)$ that
	\[\min\left\{d_W(\alpha_i, v); v \in \R^N\text{, } |v-\alpha_i|\geq \delta \right\} >0.\]
	As a result we get essentially uniform in time convergence in measure, i.e.\ for every $\delta >0$ we have
	\begin{equation}\label{uniform_convergence_in_measure}
	  \esssup_{0\leq t \leq T } \left|\left\{\left|u_\eps - \sum_{i=1}^P\chi_i \alpha_i\right| \geq \delta\right\}\right| \to 0.
	\end{equation}
	Since $u_\eps$ is continuous in time, we can replace the essential supremum by a ``true'' supremum. 
	
      \textit{Step 4: The sequence $u_\eps^2$ is equi-integrable uniformly in time.}\\
	If $p>2$, then this follows immediately from the uniform $L^p$ bound \eqref{u goes to zeros} of $u_\eps$ we proved in Lemma \ref{chain rule} by an application of the H\"older inequality:
	For any measurable set $A \subset \torus$ and any $\eps>0$ we have
	\begin{equation}\label{uniform_equiintegrability}
	 \sup_{0\leq t\leq T}\int_A u_\eps^2(t,x) dx \leq \sup_{0\leq t\leq T} |A|^\frac{2}{p'} \left(\int |u_\eps|^p dx\right)^\frac{2}{p} \lesssim |A|^\frac{2}{p'} \left(1 + E_\eps(u_\eps^0)\right)^\frac{2}{p}.
	\end{equation}
	As $E_\eps(u_\eps^0)$ is bounded uniformly in $\eps$, we get the statement.
	
	If $p=2$ we get some slightly better integrability from a Sobolev embedding:
	Let $G(u):= \left(|u| - R \right)_+^2$, where $R>0$ is the radius from the growth condition \eqref{growth} of $W$. 
	This function is $C^1$ with 
	\[\partial_u G(u) = 2\left(|u|-R\right)_+\frac{u}{|u|} \chi_{|u|>R}\]
	and thus satisfies the same bounds as $\phi_i$, see \eqref{growth_of_phi} and \eqref{Lipschitz estimate}, namely
	\[G(u) \leq  |u|^2\text{, } |\partial_u G(u)|\leq |u|.\]
	Consequently, we can use the same approximation argument as in Lemma \ref{chain rule} to see that
	\[\sup_{\eps>0} \sup_{0\leq t\leq T} ||G\circ u_\eps(t)||_{W^{1,1}} <\infty.\]
	The Sobolev embedding theorem can thus be applied to conclude
	\[\sup_{\eps>0} \sup_{0\leq t\leq T} ||G\circ u_\eps(t)||_{L^{\frac{d}{d-1}}}<\infty.\]
	Recalling the definition of $G$ we see that this implies
	\[\sup_{\eps>0} \sup_{0\leq t\leq T} ||u_\eps(t)||_{L^{2\frac{d}{d-1}}}<\infty,\]
	from which we deduce the necessary equi-integrability of $|u_\eps|^2$ as before.
	
      \textit{Step 5: The sequence $u_\eps$ converges in $C([0,T];L^2(\torus))$.}\\
	Essentially, we wish to exploit the fact that convergence in measure and equi-integrability are equivalent to convergence in $L^1$.
	However, since we want the convergence to be uniform in time and instead of $L^1$ convergence we want $L^2$ convergence in space, we quickly reproduce the argument.
	
	For any cut-off $M>0$ we can split the integral
	\[ \int |u_\eps - u|^2 dx    = \int_{\{|u_\eps - u| \geq M\}} |u_\eps -u|^2 dx + \int_{\{|u_\eps - u|< M\}} |u_\eps - u|^2 dx.\]
	The first term on the right-hand side satisfies 
	\[\sup_{0\leq t \leq T} \int_{\{|u_\eps - u| \geq M\}} |u_\eps - u|^2 \lesssim  \sup_{0\leq t \leq T}\int_{\{|u_\eps - u| \geq M\}} (|u_\eps|^2 +1)\,dx \to 0\quad \text{as } \eps\to 0\]
	by applying uniform convergence in measure \eqref{uniform_convergence_in_measure} and uniform equi-integrability \eqref{uniform_equiintegrability}.
	For every $\delta >0$ the second term on the right hand side can be estimated by
	\[
	  \sup_{0\leq t \leq T} \int \min\left\{|u_\eps - u|^2, M^2\right\} dx
	  \leq  \sup_{0\leq t \leq T}  \Lambda^d \delta^2 + \left|\left\{\left|u_\eps -u\right| >\delta\right\}\right|M^2 \to \Lambda^d \delta^2, \quad \text{as } \eps \to 0.
	\]
	Taking first  $\eps\to0$ and then $\delta\to0$ we have indeed
	\[\lim_{\eps \to 0} \sup_{0\leq t \leq T}  \int |u_\eps - u|^2 dx =0.\qedhere\]      
      \end{proof}

      \begin{proof}[Proof of Proposition \ref{prop_normal_velocities}]
	The strategy is the following:
	\begin{enumerate}
	\item We prove the easier fact $\partial_t (\phi_i\circ u) \ll E(\bullet, u) dt$.
	\item We replace $\phi_i \circ u$ with $u$, i.e.\ we prove $\partial_t u \ll E(\bullet, u) dt$, using a suitable localization of Step 4 of the proof 
	of Proposition \ref{prop_comp}, i.e.\ the Fleming-Rishel coarea formula.
	\item We prove that $\partial_t \chi_i$ is singular to the ``wrong'' parts of $E(\bullet, u)dt$ in order to replace the right-hand side with
	$|\nabla \chi_i|dt$.
	\end{enumerate}
	Equation \eqref{v=dtX} immediately follows.
	
      \textit{Step 1: For all $1\leq i \leq P$ we have $\partial_t (\phi_i\circ u) \ll E(\bullet, u) dt$ and the corresponding density is square-integrable w.r.t.\ $E(\bullet, u)dt$.}\\
	We localize with a smooth test function $\zeta\in C_0^\infty((0,T)\times \torus; \R^{1+d})$ and use the chain rule \eqref{chain rule}, the Lipschitz estimate \eqref{Lipschitz estimate} and the Cauchy-Schwarz inequality to obtain
	\begin{align}\label{proof of V2 1}
	  \int_0^T \int \partial_t \phi_i(u_\eps)  \zeta \,dx\,dt
	  \leq \left( \int_0^T \int \eps \left|\partial_t u_\eps\right|^2dx\,dt\right)^{\frac{1}{2}}
	  \left(\int_0^T\int  \zeta^2 \frac{2}{\eps} W(u_\eps)\,dx\,dt\right)^{\frac{1}{2}}.
	\end{align}
	By the convergence \eqref{phi eps to phi} of the composition and the equipartition of energy \eqref{equipartition} we can pass to the limit in this inequality and obtain
	\begin{equation}\label{proof of V2 2}
	\int_0^T \int \phi_i(u) \partial_t  \zeta \,dx\,dt \leq  \left(\liminf_{\eps\downarrow0}  \int_0^T \int \eps \left|\partial_t u_\eps\right|^2dx\,dt\right)^{\frac{1}{2}}
	\left(\int_0^T E_\eps (\zeta^2, u) dt\right)^{\frac{1}{2}}.
	\end{equation}
	By equation \eqref{energy-dissipation equality} the first factor on the right-hand side is controlled by $\sqrt{E_0}$.
	From this we see that indeed $\left| \partial_t(\phi_i\circ u)\right| \ll E(\bullet,u) dt$ and by taking the 
	supremum over the test functions $\zeta$ we see that the density is square-integrable.
	
      \textit{Step 2: We have $d_i |\partial_t \chi_i| \leq |\partial_t (\phi_i \circ u)|$ where $d_i:= \min_{1\leq j \leq P, i\neq j} d_W(\alpha_i, \alpha_j)$.}\\
	Basically, we want to use the argument of Step 4 in the proof of Proposition \ref{prop_comp} for the partial derivative $\partial_t \chi_i$.
	This can be done by combining the slicing theorem, cf.\ Theorem 3.103 in \cite{ambrosio2000functions},  and with the previous argument at almost each point $x \in \torus$, which leads to
	\[d_i |\partial_t \chi_i|(U) \leq |\partial_t (\phi_i \circ u)|(U)\]
	for all open sets $U \subset [0,T]\times \torus$.
	This implies that for all $\xi \in C_c([0,T]\times\torus; [0,\infty)$ we have the inequality
	\[d_i |\partial_t \chi_i|(\xi) \leq |\partial_t (\phi_i \circ u)|(\xi):\]
	Indeed, we can approximate $\xi$ by constants on sets whose boundaries are negligible w.r.t.\ the measures on both sides.
	We thus get that for all $1\leq i \leq P$ we have $\partial_t \chi_i \ll E(\bullet, u) dt$ and the corresponding density $V_i$ satisfies $V_i \in L^2\left(E(\bullet, u) dt\right)$.
	
      \textit{Step 3: We have that $|(\partial_t,\nabla) \chi_i|$ and $\frac{1}{2}\left(|\nabla \chi_j|_d + |\nabla \chi_k|_d - |\nabla(\chi_j + \chi_k)|_d\right)dt$ are singular for all pairwise different $1\leq i,j,k \leq P$.}\\
	For a characteristic function $\chi : [0,T]\times \torus \to \R$ we write $|\nabla \chi|_{d+1}$ for the total variation in time and space of the partial spacial derivatives and $|\nabla \chi|_d$ for the total variation the spacial derivatives in space defined almost everywhere in time.
    %

	According to  Theorem 4.17 in \cite{ambrosio2000functions} one can decompose $\supp |(\partial_t,\nabla) \chi_i|$ into the pairwise disjoint sets 
	$\tilde \Sigma_{i,l}:= \partial_\ast \tilde\Omega_i \cap \partial_\ast \tilde\Omega_l$, $1\leq l\leq P$, which are the intersections of the reduced boundaries in time and space.
	The exceptional sets are $\mathcal{H}^d$-negligible and hence can be ignored in all the derivatives $ |(\partial_t,\nabla) \chi_{m}|$, $1\leq m \leq P$.
	Thus we only have to prove that
	\[\frac{1}{2}\left(|\nabla \chi_j|_d + |\nabla \chi_k|_d - |\nabla(\chi_j + \chi_k)|_d\right)dt \left(\tilde \Sigma_{il}\right) =0\]
	for all $1\leq l \leq P$.
	
	Since $j,k\neq i$ and the interfaces are pairwise disjoint we have that
	\[\left|(\partial_t, \nabla)\chi_j\right|\left(\tilde \Sigma_{il}\right)=0 \text{ or } \left|(\partial_t, \nabla)\chi_k\right|\left(\tilde \Sigma_{il}\right)=0.\]
	In the first case we have, since restriction commutes with the total variation,
	\[\frac{1}{2}\left(|\nabla \chi_j|_{d+1} + |\nabla \chi_k|_{d+1} - |\nabla(\chi_j + \chi_k)|_{d+1}\right)\left(\tilde \Sigma_{il}\right) = \frac{1}{2}\left(|\nabla \chi_k|_{d+1}\left(\tilde \Sigma_{il}\right) -|\nabla \chi_k|_{d+1}\left(\tilde \Sigma_{il}\right)\right) = 0.\]
	The analogous argument gives the same result in the second case.
	Finally, a straightforward generalization of  Theorem 3.103 in \cite{ambrosio2000functions} to higher dimensional slicings implies
	\[|\nabla \chi_l|_{1+d} = |\nabla \chi_l|_d\, dt,\]
	which proves the claim.
	
	\textit{Step 4: Conclusion of the $L^2$-estimate.}\\
	Since $|\partial_t \chi_i| \leq |(\partial_t, \nabla)\chi_i|$ as measures we get from Step 2 and Step 3 that $|\partial_t \chi_i| \ll |\nabla \chi_i|_d \, dt$.
	Step 3 also allows to replace $E(\bullet ,u) \, dt$ by $|\nabla \chi_i|_d \, dt$ in the $L^2$-estimate.

	We once more point out that we did not use the Allen-Cahn Equation \eqref{allen cahn} apart from the energy dissipation estimate \eqref{energy-dissipation equality}.
      \end{proof}

      \begin{proof}[Proof of Lemma \ref{lem impl conv ass}]
      The proof is already contained in \cite{laux2015convergence} and \cite{luckhaus1989gibbs}. For the convenience of the reader we reproduce the arguments here.
      
      \textit{Step 1: Localization in time.}\\
      We first show that the integrated assumption of the convergence of the energies
	(\ref{conv_ass}) and the $\Gamma$-convergence of $E_\eps$ to $E$ already imply the pointwise convergence \eqref{conv ass pw in t} -- at least up to a further subsequence.
	We will prove
	\begin{equation}\label{conv ass L1}
	  \lim_{\eps\downarrow 0} \int_0^T \left|  E_\eps(u_\eps) -E(\chi)\right| dt =0,
	\end{equation}
	which after passage to a subsequence clearly implies \eqref{conv ass pw in t}.
	
	To convince ourselves of \eqref{conv ass L1} we rewrite the integral as
	\begin{equation*}
	  \int_0^T \left|  E_\eps(u_\eps) -E(\chi) \right| dt =  \int_0^T \left(  E_\eps(u_\eps) -E(\chi) \right) dt + 2\int_0^T \left(  E_\eps(u_\eps) -E(\chi) \right)_- dt.
	\end{equation*}
	The first right-hand side integral vanishes as $\eps\downarrow0$ by \eqref{conv_ass}. 
	By the lower semi-continuity part of the $\Gamma$-convergence of $E_\eps$ to $E$, see \cite{baldo1990minimal}, and by the convergence \eqref{u eps to chi} of $u_\eps$ to $u$
	the integrand of the second right-hand side term tends to zero pointwise a.e.\ in $(0,T)$. By Lebesgue's dominated convergence also the integral vanishes in the limit
	and we proved \eqref{conv ass L1}.
	
	\textit{Step 2: Localization in space.}\\
	We claim that the convergence \eqref{conv ass pw in t} of the energies implies
	\begin{equation}\label{conv ass local}
	  \lim_{\eps\downarrow0} E_\eps(\zeta, u_\eps) = E(\zeta,u)\quad \text{for }\ae \; 0\leq t \leq T\; \text{and all } \zeta\in C^\infty(\torus).
	\end{equation}
	Indeed, if we assume that w.l.o.g.\ by linearity $0 \leq \zeta \leq 1$, using the $\liminf$-inequality of the  
	$\Gamma$-convergence on the domains $\{\zeta>s\}$ and the layer cake representation $\zeta = \int_0^1 \chara_{\{\zeta>s\}} \,ds$ we obtain the inequality
	\begin{align*}
	  E(\zeta,u) \leq \liminf_{\eps\downarrow0} E_\eps(\zeta,u_\eps). 
	\end{align*}
	But the same argument works for $0\leq 1-\zeta \leq 1$ instead of $\zeta$ and by the convergence \eqref{conv ass pw in t} we have
	\[
	  E(u) - E(\zeta,u) = E(1-\zeta,u) \leq \liminf_{\eps\downarrow0} E_\eps(1-\zeta,u_\eps) 
	  \overset{\eqref{conv ass pw in t}}{=} E(u) - \limsup_{\eps\downarrow0} E_\eps(\zeta,u_\eps),
	\]
	which is the inverse inequality and thus \eqref{conv ass local} follows.
	
	\textit{Step 3: Equipartition of energy.}\\
	Now let us turn to \eqref{equipartition}. 
	First  we claim that \eqref{equipartition} reduces to
	\begin{equation}\label{equipartition simplified}
	  \int \sqrt{2W(u_\eps)} \left| \nabla u_\eps\right| dx \to E(u).
	\end{equation}
	
	Indeed, setting 
	$
	  a_\eps^2 := \frac{\eps}2 \left| \nabla u_\eps \right|^2$ and $b_\eps^2 := \frac1\eps W(u_\eps)$,
	using $a_\eps^2 - b_\eps^2 = \left( a_\eps +  b_\eps\right) \left( a_\eps -  b_\eps\right)$ and Cauchy-Schwarz
	\begin{align*}
	\int \zeta \left| a_\eps^2 -  b_\eps^2 \right| dx
	\leq \|\zeta\|_\infty \left(\int  \left( a_\eps +  b_\eps \right)^2 dx \right)^{\frac12} 
	\left(\int \left( a_\eps -  b_\eps \right)^2 dx \right)^{\frac12}.
	\end{align*}
	Since $\left( a_\eps +  b_\eps \right)^2 \lesssim a_\eps^2 + b_\eps^2 $ the first right-hand side integral stays bounded in the limit $\eps\downarrow0$
	and it is enough to prove that the second right-hand side integral vanishes as $\eps\downarrow0$.
	Expanding the square and using the definition of $a_\eps$ and $b_\eps$ we see that the limit of the second right-hand side integral is equal to
	\[
	  E_\eps(u_\eps) -  \int \sqrt{2W(u_\eps)} \left| \nabla u_\eps\right| dx 
	\]
	and indeed the proof of \eqref{equipartition} reduces to \eqref{equipartition simplified}.
	
	We conclude by proving \eqref{equipartition simplified}.
	By lower semi-continuity and Young's inequality for any cutoff $ 0 \leq \eta \leq 1$ and any $1\leq i \leq P$ we get
	\begin{align*}
	  \int \eta \left| \nabla \left( \phi_i\circ u\right) \right| 
	  \leq \liminf_{\eps\downarrow0} \int \eta \left| \nabla \left( \phi_i\circ u_\eps\right) \right| dx
	  &\leq \liminf_{\eps\downarrow0} \int \eta \sqrt{2W(u_\eps)} \left| \nabla u_\eps \right| dx     \\
	  &\leq \liminf_{\eps\downarrow0} E_\eps(\eta, u_\eps) 
	  \overset{\eqref{conv ass local}}{=} E(\eta,\chi).
	\end{align*}
      Using a partition of unity subordinate to the covering \eqref{covering Br} and choosing the index $1\leq i \leq P$ according to estimate \eqref{baldo E is supremum} we conclude.

      \end{proof}

    \section{Convergence}\label{sec:conv}



    In Section \ref{sec:comp} we proved that the solutions $u_\eps$ of the Allen-Cahn Equation \eqref{allen cahn} are pre-compact.
    In this section we  pass to the limit in the Allen-Cahn Equation \eqref{allen cahn} and prove that the limit moves by mean curvature.
    Since this section is the core of the paper, we give a short idea of the proof and then pass to the rigorous derivation in the subsequent parts, first for the curvature term, and afterwards for the velocity term.

    \subsection{Idea of the proof}

    To illustrate the idea of our proof we give a short overview in the simpler two-phase case.
    In this setting the convergence of the curvature-term 
    \begin{equation}\label{reshetnyak 2phase}
      \lim_{\eps\downarrow0} \int_0^T \int \left( \eps \Delta u_\eps -\frac{1}{\eps} W'(u_\eps)\right)  \xi \cdot \nabla u_\eps\, dx \,dt
      = \sigma \int \nabla \xi \colon \left( Id -\nu \otimes \nu  \right)\left| \nabla \chi \right| dt 
    \end{equation}
    is by the pointwise in time convergence of the energies \eqref{conv ass pw in t} literally contained in \cite{luckhaus1989gibbs} and the only difficulty is to prove
    \begin{equation}\label{overview dtu Du}
    \lim_{\eps\downarrow0} \int_0^T \int \partial_t u_\eps\, \xi \cdot \eps \nabla u_\eps \,dx \,dt = \sigma\int_0^T\int V \, \xi\cdot \nu \left| \nabla \chi \right| dt.
    \end{equation}
    Since $\partial_t u_\eps \rightharpoonup V \left| \nabla \chi\right| dt$ and $\eps \nabla u_\eps \approx \nu$ only in a weak sense,
    we cannot directly pass to the limit in the product.
    The general idea to work around this problem is to follow the strategy of \cite{laux2015convergence}: Thinking of the test vector field $\xi$ as a localization, 
    we ``freeze'' the normal along the sequence to be the fixed direction $\nu^\ast\in \sph^{d-1}$ and estimate the error w.r.t.\ an 
    approximation of the \emph{tilt-excess}
    \begin{equation}\label{exc}
      \exc :=  \sigma \int_0^T \int  |\nu - \nu^\ast|^2 |\nabla \chi|\, dt,
    \end{equation}
    measuring the (local) flatness of the reduced boundary $\partial_\ast \Omega$ of the limit phase $\Omega = \{\chi=1\}.$
    The main difference to the work \cite{laux2015convergence} is that we measure the error w.r.t.\ the tilt-excess $\exc$ instead of the energy-excess 
    \[
      \int \left| \nabla \chi \right| - \int \left| \nabla \chi^\ast \right|, \quad \text{where } \chi^\ast \text{ is a half-space in direction }\nu^\ast.
    \]
    After a localization, De Giorgi's Structure Theorem guarantees the smallness in both cases, see Section 5 in \cite{laux2015convergence}.
    Our approximation of the tilt-excess along the sequence is
    \begin{equation}\label{exc eps}
    \exc_\eps :=  \int_0^T \int \left| \nu_\eps - \nu^\ast\right|^2\eps  \left|\nabla u_\eps \right|^2 dx\,dt,
    \end{equation}
    where $ \nu_\eps= \frac{\nabla u_\eps}{|\nabla u_\eps|}$ denotes the normal of the level sets of $u_\eps$.
    
    \medskip
    
    We will use the approximate tilt-excess to suppress oscillations of the \emph{direction}
      of the term $\eps \nabla u_\eps$ on the left-hand side of \eqref{overview dtu Du} so that we can pass to the limit in the product.
      We replace the normal $\nu_\eps$ 
      by a constant direction $\nu^\ast \in \sph^{d-1}$ and control the difference
    \begin{equation}\label{overview freeze normal}
      \int_0^T \int \partial_t u_\eps\, \xi \cdot \eps \nabla u_\eps  dx \,dt -
      \int_0^T \int \partial_t u_\eps\, \xi \cdot \left( \eps \left| \nabla u_\eps \right| \nu^\ast \right)   dx \,dt
    \end{equation}
    by the following combination of the excess and the initial energy
    \[
	\|\xi\|_\infty \left( \frac1\alpha \exc_\eps + \alpha E_\eps(u_\eps^0)\right)
    \]
    for any (small) parameter $\alpha>0$ --  an immediate consequence of Young's inequality and the energy-dissipation estimate \eqref{energy-dissipation equality}.
    It is easy to check that by the equipartition of energy \eqref{equipartition} we can replace $\eps \left| \nabla u_\eps \right|$ in the second integral in \eqref{overview freeze normal}
    by $\sqrt{2W(u_\eps)}$ up to an error
    that vanishes as $\eps\downarrow 0$:
    \begin{equation}
    \int_0^T \int \partial_t u_\eps\, \xi \cdot \left( \eps \left| \nabla u_\eps \right| \nu^\ast \right)   dx \,dt
    = \int_0^T \int \partial_t u_\eps\,\sqrt{2 W(u_\eps)}\, \xi \cdot \nu^\ast    dx \,dt +o(1).
    \end{equation}
    Identifying the nonlinear term 
    \[
     \partial_t u_\eps\,\sqrt{2 W(u_\eps)}= \partial_t \left( \phi \circ u_\eps\right)
    \]
    as the derivative of the compact quantity $\phi\circ u_\eps \to \phi \circ u$, where $\phi(u) = \int_0^u  \sqrt{2 W(s)} ds$, we can pass to the limit $\eps\downarrow0$ and obtain
    \[
      \int_0^T \int \partial_t \left( \phi\circ u_\eps\right) \xi \cdot \nu^\ast    dx \,dt \to  \sigma \int_0^T \int V  \,\xi \cdot \nu^\ast  \left| \nabla \chi \right| dt.
    \]
    As before, but now at the level of the limit, by Young's inequality we can ``un-freeze'' the normal, i.e.\ 
    replace $\nu^\ast$ by $\nu$ at the expense of 
    \[
      \|\xi\|_\infty \left( \frac1\alpha \exc + \alpha \int_0^T \int V^2 \left|\nabla \chi \right| dt \right).
    \]
    
    While in the case of \cite{laux2015convergence} the convergence assumption trivially implies the convergence of the (approximate) energy-excess, here we have to argue
    why we can pass to the limit in our nonlinear excess $\exc_\eps$ and connect it to $\exc$.
    
    Using the trivial equality $ |\nu -\nu^\ast|^2 = 2(1- \nu\cdot \nu^\ast)$ and the convergence assumption \eqref{conv_ass} this question reduces to
    \begin{equation}\label{overview nonlinear}
      \int_0^T \int \eps \left| \nabla u_\eps \right| \nabla u_\eps dx\,dt \to \int_0^T \int \nabla \left( \phi\circ u\right) dt.
    \end{equation}
    Now the argument is similar to the one before for the time derivative.
    Using again the equipartition of energy \eqref{equipartition} we can replace $\eps \left|\nabla u_\eps\right| $  by  $\sqrt{2 W(u_\eps)}$.
    Identifying the nonlinearity 
    $\sqrt{2W(u_\eps)} \nabla u_\eps = \nabla \left( \phi \circ u_\eps \right)$
    as a derivative yields the convergence of the excess.
        
        \medskip
        
    Thus we arrive at the right-hand side of \eqref{overview dtu Du} -- up to an error
    that we can handle: we localize on a scale $r>0$ so that $\exc\to 0$ as $r\downarrow0$, while the second error term stays bounded by the $L^2$-estimate \eqref{V squared}.
    We then recover the motion law \eqref{H=v}  by sending $\alpha\downarrow0$.

     \subsection{Convergence of the curvature-term}\label{sec:conv curv}
    In the two-phase case, the convergence \eqref{reshetnyak 2phase} of the curvature-term is contained in the work of Luckhaus and Modica \cite{luckhaus1989gibbs}.
    In our setting, the convergence does not follow immediately from their work. We give an extension of this result by quantifying the Reshetnyak-argument.
    
    \begin{prop}\label{multi luckhaus modica}
    Given a sequence $u_\eps \to u = \sum_i \chi_i \alpha_i$ such that the energies converge in the sense of 
    \begin{align}\label{reshetnyak multiphase conv ass}
    E_\eps(u_\eps) \to E(u).
    \end{align}
    Then also the first variations converge:
    \begin{align} 
    &\lim_{\eps\downarrow0} \int \left( \eps \Delta u_\eps -\frac{1}{\eps} \partial_u W (u_\eps)\right) \cdot \left( \xi \cdot \nabla\right) u_\eps\, dx \notag\\
    &= \frac12 \sum_{1 \leq i,j\leq \numphases} \sigma_{ij} \int \nabla \xi \colon 
    \left( Id - \nu_i \otimes \nu_i\right) \frac12 \left( \left| \nabla \chi_i \right| + \left| \nabla \chi_j \right| - \left| \nabla (\chi_i+\chi_j) \right|\right).\label{reshetnyak multiphase} 
    \end{align}
    Furthermore, we have
    \begin{equation}\label{reshetnyak PI gives estimate}
      \int \left( \eps \Delta u_\eps -\frac{1}{\eps} \partial_u W (u_\eps)\right) \cdot \left( \xi \cdot \nabla\right) u_\eps\, dx \lesssim \|\nabla\xi\|_\infty E_\eps(u_\eps).
    \end{equation}

    \end{prop}

    \begin{proof}
     Following the lines of \cite{luckhaus1989gibbs} we can rewrite the left-hand side 
      of (\ref{reshetnyak multiphase}) by integrating the first term by parts and using the chain rule for the second term. With
      Einstein's summation convention and omitting the index $\eps$ we have
      \begin{equation}\label{reshetnyak proof 1}
	\int ( \eps \partial_i \partial_i u_k -\frac{1}{\eps} \partial_k W )\, \xi_j \, \partial_j u_k\, dx 
	=  \int \big\{ -\eps \,\partial_i u_k\, \partial_i \xi_j\,  \partial_j  u_k - \eps\, \partial_i u_k\, \xi_j \, \partial_i \partial_j  u_k
	-\frac{1}{\eps} \partial_j ( W(u)) \,\xi_j \big\}\,dx.
      \end{equation}
      We can now rewrite the second term on the right-hand side and integrate by parts to see
      \begin{align*}
	-\int \eps\, \partial_i u_k\, \xi_j \, \partial_i \partial_j  u_k \,dx
	= -\int \eps\, \xi_j \, \partial_j \Big\{\frac{1}{2} \left( \partial_i   u_k\right)^2 \Big\} \,dx
	=   \int  \left( \nabla \cdot \xi \right) \frac\eps2 \left| \nabla u \right|^2 dx.
      \end{align*}
      Plugging this into (\ref{reshetnyak proof 1}) the left-hand side of (\ref{reshetnyak multiphase}) is thus equal to
      \begin{equation*}
	\int \nabla \xi \colon \left( Id - N^\eps \otimes N^\eps \right) \eps \left| \nabla u_\eps\right|^2 dx
	  + \int  \left( \nabla \cdot \xi\right) \Big( \frac1\eps W(u_\eps) - \frac\eps2\left| \nabla u_\eps\right|^2 \Big) dx,
      \end{equation*}  
      where $N^\eps := \frac{\nabla u_\eps}{|\nabla u_\eps|} \in \R^{\numphases \times d}$ and $(N^\eps \otimes N^\eps)_{ij} := \sum_k N^\eps_{ki} N^\eps_{kj} \in \R^{d\times d}$,
      a slightly non-standard definition of this symbol.
      From this we immediately obtain \eqref{reshetnyak PI gives estimate}.
      By  the equipartition of energy (\ref{equipartition}), see also Remark \ref{compactness_for_forces}, the  second integral is negligible as $\eps \to 0$ and up to another error that vanishes as $\eps\to0$ we can replace the first term by
      \begin{equation*}
	\int \nabla \xi \colon \left( Id - N^\eps \otimes N^\eps \right) \sqrt{2W(u_\eps)} \left| \nabla u_\eps\right| dx.
      \end{equation*}
      Again by the equipartition of energy (\ref{equipartition}) it is enough to prove the convergence of the nonlinear term
      \begin{align}\label{reshetnyak proof 2}
	\int A \colon  N^\eps \otimes N^\eps \sqrt{2W(u_\eps)} \left| \nabla u_\eps\right|  dx \to 
	\sum_{i,j} \sigma_{ij} \int A \colon 
    \nu_i \otimes \nu_i \frac12 \left( \left| \nabla \chi_i \right| + \left| \nabla \chi_j \right| - \left| \nabla (\chi_i+\chi_j) \right|\right)
      \end{align}
      for any smooth matrix field $A \colon \torus \to \R^{d\times d}$. By linearity we may assume w.l.o.g.\ $|A|\leq 1$.
      
      \medskip
      
      We prove \eqref{reshetnyak proof 2} using the following two claims:
      
      \smallskip
      
      \textit{Claim 1: We choose a majority phase by introducing the function $\phi = \phi_i$ for some 	arbitrary $1\leq i \leq P$ on the right-hand side of \eqref{reshetnyak proof 2}.
	      The corresponding estimate is
	      \begin{align}\label{reshetnyak proof step 1}
		\limsup_{\eps\to0} \left| \int A \colon  N^\eps \otimes N^\eps \sqrt{2W(u_\eps)} \left| \nabla u_\eps\right|  dx - 
		\int A \colon \nu_\eps \otimes \nu_\eps \left| \nabla (\phi\circ u_\eps)\right| dx\right| 
		\lesssim  E(u, \eta)-\int\eta \left| \nabla \phi(u) \right|,
	      \end{align}
	      where $\nu_\eps := \frac{\nabla \phi(u_\eps)}{|\nabla \phi(u_\eps)|} \in \R^{d}$.}
	      
	      \smallskip
      
      \textit{Claim 2: We adapt the Reshetnyak argument in \cite{luckhaus1989gibbs} to our setting by 		turning the qualitative statements there into a quantitative statement.
	      Under the assumption (\ref{reshetnyak multiphase conv ass}) we claim
	      \begin{equation}\label{reshetnyak proof step 2}
		\limsup_{\eps\downarrow0} \left| \int A \colon \nu_\eps \otimes \nu_\eps \left| \nabla( \phi \circ u_\eps)\right| dx
		- \int A \colon \nu \otimes \nu \left| \nabla (\phi \circ u) \right| \right| \lesssim E(u, \eta)-\int\eta \left| \nabla (\phi \circ u) \right|.
	      \end{equation}}
      
      In both cases we express the errors in terms of the ``mild excess''
      \begin{equation}\label{mild excess}
      E(u, \eta)-\int\eta \left| \nabla \phi(u) \right|,
      \end{equation}
      which measures the local difference of the multi-phase setting to the two-phase setting on the support of the matrix field $A$  
      approximated with a cut-off $\eta$.
      
      Decomposing an arbitrary matrix field $A$ by a partition of unity and using the localization estimate \eqref{baldo E is supremum} we obtain \eqref{reshetnyak proof 2} and thus proved the proposition.
      
      \smallskip
      
      \textit{Proof of Claim 1: Introducing a majority phase.}\\
	First we replace the matrix $N^\eps = \frac{\nabla u_\eps}{|\nabla u_\eps |}$ by $\pi_{u_\eps} N^\eps$, where 
	$\pi_{u} = \frac{\partial_u \phi}{|\partial_u \phi|} \otimes \frac{\partial_u \phi}{|\partial_u \phi|}$.
	Note that then the additional sum in the definition of the symbol $\pi_{u_\eps} N^\eps  \otimes \pi_{u_\eps} N^\eps$ collapses:
	\[
	  \left(\pi_{u} N  \otimes \pi_{u} N\right)_{ij} = \sum_{k=1}^N \left(\pi_{u} N \right)_{ki} \left(\pi_{u} N\right)_{kj} = \left(\sum_{k=1}^N \frac{(\partial_u\phi)_k^2}{|\partial_u\phi|^2} \right)
	  \frac{(\partial_u \phi \nabla u)_{i}}{|\partial_u \phi| |\nabla u|}   \frac{(\partial_u \phi \nabla u)_{j}}{|\partial_u \phi| |\nabla u|}.
	\]
	Furthermore, using the chain rule of Ambrosio and Dal Maso, Lemma \ref{chain rule}, we see
	\begin{align*}
	  A\colon \left( \pi_{u_\eps} N^\eps \right) \otimes\left( \pi_{u_\eps} N^\eps \right) 
	  = |\pi_{u_\eps} N^\eps|^2 A \colon \nu_\eps  \otimes \nu_\eps.
	\end{align*}
	
	Two errors arise in (\ref{reshetnyak proof step 1}). The first error when replacing $N^\eps$ by $\nu_\eps$ and the second when replacing $\sqrt{2W(u_\eps)}|\nabla u_\eps|$ by
	$|\nabla (\phi\circ u)|$.
	
	The first error when introducing the projection $\pi_u$ is bounded by
	\begin{align}
	  \int \eta \left|\left( Id- \pi_{u_\eps} \right)  N^\eps \right|^2 \sqrt{2W(u_\eps)} \left| \nabla u_\eps\right| 
	  +\eta \left( 1- |\pi_{u_\eps}N^\eps|^2 \right)^2 \sqrt{2W(u_\eps)} \left| \nabla u_\eps\right| dx.\label{reshetnyak proof 3}
	\end{align}
	Since multiplication by $\pi_u$ is an orthogonal projection in matrix-space and $\left|\frac{\partial_u \phi}{|\partial_u\phi|}N^\eps\right|=|\pi_u N^\eps| \leq 1$ we have
	\begin{align*}
	  \left|\left( Id- \pi_{u_\eps} \right) N^\eps \right|^2
	  = \left|  N^\eps \right|^2 - \left|\pi_{u_\eps}  N^\eps \right|^2 
	  \lesssim & 1 - \left|\pi_{u_\eps}  N^\eps \right| = 1-\left|\frac{\partial_u \phi}{|\partial_u\phi|}N^\eps\right|.
	\end{align*}
	Multiplying this inequality with $\sqrt{2W(u_\eps)} \left| \nabla u_\eps\right|$ and using the Lipschitz estimate for $\phi\circ u$ \eqref{Lipschitz estimate} we see that
	\begin{align*}
	\left|\left( Id- \pi_{u_\eps} \right)  N^\eps \right|^2 \sqrt{2W(u_\eps)} \left| \nabla u_\eps\right| 
	&\leq \left(1-\left|\frac{\partial_u \phi}{|\partial_u\phi|}N^\eps\right|\right)\sqrt{2W(u_\eps)} \left| \nabla u_\eps\right| \\
	&\leq \sqrt{2W(u_\eps)}
	  \left|  \nabla u_\eps \right| -\left|\partial_u\phi(u_\eps)  \nabla u_\eps \right|.
	\end{align*}

      Plugging this into (\ref{reshetnyak proof 3}) and using the Ambrosio-Dal Maso chain rule \eqref{Dphi} again, we see that the error is controlled by
      \begin{align*}
	  E_\eps(u_\eps,\eta) - \int \eta \left| \nabla (\phi\circ u_\eps) \right| dx.
      \end{align*}
	  
      By the convergence of the energies (\ref{reshetnyak multiphase conv ass}) and lower semi-continuity of the total variation we can pass to the limit $\eps \to0$ in this expression and obtain
      the upper bound
      \begin{equation*}
	E(u,\eta) - \int \eta  \left| \nabla (\phi\circ u) \right|.
      \end{equation*}

      Finally, we turn to the second error, when substituting $\sqrt{2W(u_\eps)} \left| \nabla u_\eps\right| $ by $\left| \nabla (\phi\circ u_\eps) \right| $ in (\ref{reshetnyak proof step 1}).
      Since $|\nabla (\phi \circ u_\eps)| \leq |\partial_u\phi| |\nabla u_\eps|  \leq \sqrt{2W(u_\eps)} \left| \nabla u_\eps\right|$,
      by Young's inequality this second error is estimated by
      \begin{align*}
	\int \eta \left| \sqrt{2W(u_\eps)} \left| \nabla u_\eps\right| -\left| \nabla (\phi\circ u_\eps) \right|\right| dx
	  = \int  \eta \left( \sqrt{2W(u_\eps)} \left| \nabla u_\eps\right| -  \left| \nabla (\phi\circ u_\eps) \right| \right)dx,
      \end{align*}
      which by the equipartition \eqref{equipartition} and Remark \ref{compactness_for_forces} again passes to the limit as before and thus proves (\ref{reshetnyak proof step 1}).
      
      \smallskip
      
      \textit{Proof of Claim 2: A quantitative Reshetnyak-argument for $\phi\circ u$.}\\
	We could pass to the limit in the nonlinear expression $\int A \colon \nu\otimes \nu \left| \nabla \left(\phi\circ u\right)\right|$ by the classical Reshetnyak argument if we knew that the mass $\int \left| \nabla \left(\phi\circ u\right)\right|$ converged.
	In our case we unfortunately do not know if the total variation for each 
	$\phi_i \circ u$ converges, but we can make the error small by localizing.

	Our argument for (\ref{reshetnyak proof step 2}) can be regarded as a quantitative analogue of the classical Reshetnyak-argument \cite{reshetnyak1968weak}, see also \cite{luckhaus1989gibbs}.
	
	By Banach-Alaoglu and a disintegration result for measures we can find a measure $\mu$ on $\torus$ and a family of probability measures $\{p_x\}_{x\in \torus}$ on $\sph^{d-1}$ such that
	\begin{equation}\label{reshetnyak def px mu}
	  \int \zeta(x,\nu_\eps)  \left| \nabla (\phi\circ u_\eps)\right|  dx \to \iint \zeta(x,\tilde \nu) \,dp_x(\tilde \nu) \, d\mu(x)
	\end{equation}
	for all $\zeta \in C(\torus \times \sph^{d-1})$ -- at least after passage to a subsequence. But since we will identify the limit we may pass to subsequences. 
	In particular we have
	\begin{equation}\label{reshetnyak def px mu with A}
	  \int A \colon \nu_\eps \otimes \nu_\eps \left| \nabla (\phi\circ u_\eps)\right| dx \to \int A(x)\colon \int \tilde\nu \otimes \tilde \nu \,dp_x(\tilde \nu) \, d\mu(x).
	\end{equation}
	Our aim is to prove that -- up to the ``mild excess'' (\ref{mild excess}) -- the right-hand side of (\ref{reshetnyak def px mu with A}) is equal to
	\[
	\int A \colon \nu \otimes \nu \left| \nabla (\phi \circ u) \right|.
	\]

	On the one hand, by the lower semi-continuity of the total variation and (\ref{reshetnyak def px mu}) with $\zeta(x,\nu) = \eta(x) \geq 0$
	\begin{equation}\label{Dphi leq mu}
	\int \eta | \nabla (\phi \circ u)|  \leq \liminf_{\eps\downarrow0}  \int \eta | \nabla (\phi\circ u_\eps)| \,dx = \iint \eta  \, d\mu,
	\end{equation}
	i.e.\  $| \nabla (\phi \circ u)|$ is dominated by $\mu$. 
	
	On the other hand, by the assumption (\ref{reshetnyak multiphase conv ass}) the measure $\mu$ 
	is dominated by the energy. Indeed, for any $\eta\geq 0$ we have by Young's inequality
	\begin{equation}\label{mu leq E}
	  \int \eta \, d\mu = \lim_{\eps\downarrow0} \int \eta \left| \nabla (\phi \circ u_\eps)\right|  dx 
	  \leq \liminf_{\eps\downarrow0} E_\eps(u_\eps,\eta) = E(\chi,\eta).
	\end{equation}

	Using
	$
	  \left|\tilde \nu \otimes \tilde \nu - \nu\otimes \nu\right| \leq 2 \left| \tilde \nu - \nu\right|
	$
	and the relation (\ref{Dphi leq mu}) between the measures $|\nabla (\phi\circ u)| $ and $\mu$ we see
	\begin{align*}
	  \left|\int A\colon \int \tilde \nu \otimes \tilde \nu \, dp_x(\tilde \nu) \, d\mu - \int A \colon \nu\otimes \nu \left| \nabla (\phi \circ u)\right| \right|
	  \lesssim &\int \eta  \left( d\mu -\left| \nabla (\phi \circ u)\right|  \right)\\
	  &+ \int \eta \int \left| \nu -\tilde \nu \right| dp_x(\tilde \nu) \left| \nabla (\phi \circ u) \right|.
	\end{align*}
	By (\ref{mu leq E}) the first right-hand side term is estimated by the ``mild excess'' (\ref{mild excess}).
	
	We are left with proving 
	\begin{equation}\label{reshetnyak nu vs expectation of p}
	  \int \eta \int \left| \nu -\tilde \nu \right| dp_x(\tilde \nu) \left| \nabla (\phi \circ u) \right| 
	  \lesssim E(\chi,\eta) - \int \eta \left| \nabla (\phi\circ u)\right|.
	\end{equation}
	
	But distributional convergence of
	$\nabla (\phi \circ u_\eps)$ towards $\nabla (\phi \circ u)$ and (\ref{reshetnyak def px mu}) with $\zeta(x,\tilde\nu) = \xi(x) \cdot \tilde \nu$ 
	yield an equality for the linear term
	\begin{align}
	\int \xi \cdot \nu \left| \nabla (\phi \circ u)\right| = \int \xi \cdot \nabla (\phi \circ u) 
	&= \lim_{\eps\downarrow0} \int \xi \cdot \nabla (\phi \circ u_\eps)\,dx \notag\\
	&= \lim_{\eps\downarrow0} \int \xi \cdot \nu_\eps \left| \nabla (\phi \circ u_\eps)\right|dx
	\overset{(\ref{reshetnyak def px mu})}{=} \int \xi \cdot \int \tilde \nu \, dp_x(\tilde \nu) \, d\mu\label{varifold vs current convergence}
	\end{align}
	for any smooth test vector field $\xi \colon \torus \to \R^d$. 
	This draws a connection between the normal $\nu$ and the expectation $\int \tilde \nu \, dp_x(\tilde \nu)$ of the measures $p_x$.
	
	Therefore for any such $\xi$ with $|\xi| \leq \eta$ we get
	\begin{align*}
	  \int \xi \cdot \int \left( \nu -\tilde \nu \right) dp_x(\tilde \nu) \left| \nabla (\phi \circ u) \right| 
	  &\overset{(\ref{varifold vs current convergence})}{=}
	  \int \xi \cdot \int \tilde \nu  dp_x(\tilde \nu) \left( d\mu -\left| \nabla (\phi \circ u) \right| \right)\\
	  &\overset{(\ref{Dphi leq mu})}{\leq} \|\xi \|_\infty \left(\int \eta \, d\mu - \int \eta \left| \nabla (\phi \circ u) \right|\right)
	\end{align*}
	and after taking the supremum over all such $\xi$ we discover
	\begin{align}
	  \int \eta \int \left| \nu -\tilde \nu \right| dp_x(\tilde \nu) \left| \nabla (\phi \circ u) \right| 
	  \leq \int \eta \, d\mu - \int \eta \left| \nabla (\phi \circ u) \right|. 
	\end{align}
	Finally, notice that another application \eqref{mu leq E} proves the claim (\ref{reshetnyak nu vs expectation of p}).
    \end{proof}

    \begin{rem}
      The quantitative Reshetnyak argument \eqref{reshetnyak proof step 2} holds also for any other Lipschitz continuous function $f(x,\tilde \nu)$ on $\sph^{d-1}$ instead of 
      $A(x) \colon \tilde \nu \otimes \tilde \nu$.
    \end{rem}
    \subsection{Convergence of the velocity-term}
    As in the proof of convergence in the two-phase case our main tool will be a suitable tilt-excess.
    However, because $\nabla u_\eps$ now describes the direction of change both in physical space and in state space, some care needs to be taken in defining such an excess.
    It is apparent that the limiting equation only sees the direction of change in physical space explicitly.
    In contrast, the change of direction in state space only enters implicitly through the surface tensions, which are the lengths of geodesics connecting the wells.
    It is therefore natural to define an approximate tilt-excess which only fixes the change of direction in physical space.

    \begin{defn}
      Let $\nu^* \in \sph^{d-1}$ and $\eta \in C^\infty\left([0,T]\times\torus;[0,1]\right)$.
      For $\eps >0$ and a function $u_\eps \in W^{1,2}([0,T]\times\torus;\R^n)$ the localized tilt-excess of the $i$-th phase, $1\leq i \leq N$, is given by
      \begin{equation}\label{exc eps mult}
	\exc^i_\eps(\nu^*;\eta,u_\eps) := \int_0^T\int \eta \frac{1}{\eps} \left|\eps \nabla u_\eps + \partial_u \phi_i (u_\eps)\otimes \nu^*\right|^2 dx dt.
      \end{equation}
      In the limit $\eps =0 $ and for a partition $\chi_i = \chara_{\Omega_i} \in BV\left([0,T]\times\torus;\{ 0,1\}\right)$ with $\sum_i \chi_i =1$ we define the tilt-excess
      for $1\leq i,j \leq P$, $i\neq j$, to be
      \begin{align}\label{exc mult}
	\exc^{ij}(\nu^*;\eta,u) := &  \int_0^T \int \eta \left|\nu_i - \nu^\ast\right|^2 \left| \nabla \chi_i \right| dt
	+\int_0^T \int \eta \left|\nu_j + \nu^\ast\right|^2 \left| \nabla \chi_j \right| dt
	+  \sum_{k \notin \{i,j\}} \int_0^T\int \eta  \left| \nabla \chi_k\right| dt, 
      \end{align}
      where $u = \sum_{1\leq i \leq N} \alpha_i \chi_i$ and $\nu_i$, as throughout the paper, is the inner normal of $\Omega_i$.
    \end{defn}

    Note that the limiting excess measures two things:
    Firstly, the last term measures whether mostly the interface between the $i$-th and the $j$-th phase is present.
    Secondly, the first two terms measure how close the interface is to being flat.

    A subtle point in the definition is that $\chi_i$ falls while moving out of the corresponding phase, while $\phi_i$ grows.
    Hence their differentials have opposite directions.
    We choose $\nu^*$ to be the approximate inner normal of $\chi_i$, which leads to the positive sign in $ \exc^i_\eps$ and the second term in $\exc^{ij}$ and the negative one in the first term in $\exc^{ij}$.
    For a similar reason the limiting excesses are not symmetric in $i$ and $j$.
    Instead we have $\exc^{ij}(\nu^*;\eta,u) = \exc^{ji}(-\nu^*;\eta,u)$.

    We first make sure that we can use $\exc^{ij}(\nu^*;\eta,\chi)$ to asymptotically bound $\exc^i_\eps(\nu^*;\eta,u_\eps)$.

    \begin{lem}\label{lem: conv exc mult}
    Let $u^\eps$ satisfy the a priori estimate \eqref{generalized_energy_dissipation} and the convergence assumption \eqref{conv_ass}.
    Then for every $1\leq i,j\leq P$, $i\neq j$, $\nu^* \in \sph^{d-1}$ and $\eta \in C^\infty([0,T]\times \torus;[0,1])$ we have
    \begin{equation}\label{conv exc mult}
      \limsup_{\eps \to 0} \exc^i_\eps(\nu^*;\eta,u_\eps) \lesssim  \exc^{ij}(\nu^*;\eta,\chi).
    \end{equation}
    \end{lem}

     Using this estimate, as in the two-phase case before, we prove
    \eqref{overview dtu Du} up to an error controlled by the tilt-excess \eqref{exc mult}.

    \begin{prop}\label{multi prop dt nu}
      Given $u^\eps$ satisfying the a priori estimate \eqref{generalized_energy_dissipation} and the convergence assumption \eqref{conv_ass}, there exists a finite Radon measure $\mu$ on $[0,T]\times \torus$, such that
      for any $1\leq i,j\leq P$, $i\neq j$, any parameter $\alpha>0$, any direction $\nu^*\in \sph^{d-1}$
      and any test vector field $\xi\in C_0^\infty((0,T)\times \R^d,\R^d) $ we have
      \begin{align}
	\limsup_{\eps\downarrow0} & \left|\int_0^T \int \eps (\xi \cdot \nabla ) u_\eps \cdot \partial_t u_\eps\,dx \,dt -
	\sigma_{ij}\int_0^T \int \xi \cdot \nu_i  V_i \frac12 \left( \left| \nabla\chi_i\right|+\left| \nabla\chi_j\right| - \left| \nabla(\chi_i+\chi_j)\right|\right)dt \right|\notag \\ 
	& \lesssim  \|\xi\|_\infty \left(\frac{1}{\alpha} \exc^{ij}( \nu^*;\eta, u)  + \alpha \mu(\eta) \right).\label{dtu nu}
      \end{align}
      Here $\eta \in C^\infty([0,T]\times \R^d)$ is a smooth cut-off for the support of $\xi$, i.e.\ $\eta \geq 0$ and $\eta \equiv 1$ on $\supp \xi$.
    \end{prop}

    \begin{proof}[Proof of Lemma \ref{lem: conv exc mult}]
    Expanding the square and exploiting that $|\nabla (\phi\circ u_\eps)| \leq \sqrt{2W(u_\eps)}$  we see that
    \[\exc^i_\eps(\nu^*;\eta,u_\eps) \leq \int_0^T\int \eta\left(\eps |\nabla u_\eps|^2 + \frac{2}{\eps} W(u_\eps) + 2 (\nu^* \cdot \nabla) u_\eps \cdot \partial_u \phi_i(u_\eps) \right)dxdt.\]
    By the chain rule \eqref{Dphi} we can rewrite the last term as
    \[(\nu^* \cdot \nabla) u_\eps \cdot \partial_u \phi_i(u_\eps) = \nu^* \cdot \nabla( \phi_i \circ u_\eps).\]
    Thus we see using the convergence assumption \eqref{conv_ass} and the convergence \eqref{phi eps to phi} of $\phi_i \circ u_\eps$ to $\phi_i \circ u$ that
    \begin{align}\label{exc lim proc}
      \limsup_{\eps \to 0} \exc^i_\eps(\nu^*;\eta,u_\eps) 
      \leq &\limsup_{\eps \to 0} 2 \int_0^T \int  \eta \left( e_\eps(u_\eps) + \nu^* \cdot \nabla( \phi_i \circ u_\eps) \right) dx\,dt\\
      =& 2 \int_0^T \left(E(\eta, u) +  \nu^* \cdot\int   \eta  \nabla( \phi_i \circ u) \right)dt. \notag
    \end{align}

    The second term can be rewritten as
    \[\nu^* \cdot \nabla(\phi_i \circ u) = \nu^* \cdot \sum_{1\leq k\leq P} \sigma_{ik} \nabla \chi_k 
    \leq \sigma_{ij} \,\nu^\ast \cdot \nabla \chi_j + \sum_{k\notin\{i,j\}} \sigma_{ik} \left| \nabla \chi_k\right|, \]
      while the first one can be estimated by
      \[
	E(\eta,u) \leq \sigma_{ij} \int \eta \left| \nabla \chi_j \right| + C \sum_{k\notin\{i,j\}}  \int \eta \left| \nabla \chi_k\right|
      \]
      for some constant $C<\infty$ only depending on $\max_{ij} \sigma_{ij}$. Thus we can asymptotically bound the excess by
      \begin{align*}
	\limsup_{\eps \to 0} \exc^i_\eps(\nu^*;\eta,u_\eps) \leq  \sigma_{ij} \int_0^T \int \eta \,2 \left( 1 + \nu_j \cdot\nu^\ast \right) \left| \nabla \chi_j\right|dt
	+ C \sum_{k\notin\{i,j\}}  \int_0^T \int \eta \left| \nabla \chi_k\right| dt.
      \end{align*}
      Since $2 \left( 1 + \nu_j \cdot\nu^\ast \right)= \left| \nu_j + \nu^\ast \right|^2$ in particular \eqref{conv exc mult} holds.
      Note that we symmetrized the multi-phase excess \eqref{exc eps} w.r.t.\ the two majority phases $\Omega_i$ and $\Omega_j$ which means we added an extra (nonnegative) term.
    \end{proof}

    \begin{proof}[Proof of Proposition \ref{multi prop dt nu}]
    \textit{Step 1: Replacing $\nabla u_\eps$ with $\partial_u\phi_i (u_\eps) \otimes \nu^*$.}\\
      Using the tilt-excess \eqref{exc eps mult} and Young's inequality we see
      \begin{equation}\label{dtu step1}
    \left|\int_0^T \int \left(\eps (\xi \cdot \nabla ) u_\eps + \xi\cdot \nu^* \partial_u \phi_i(u_\eps) \right)\cdot \partial_t u_\eps\,dx \,dt\right| 
    \lesssim \|\xi\|_\infty\left(\frac{1}{\alpha} \exc^i_\eps(\nu^*;\eta,u_\eps) + \alpha \int_0^T \int \eta \eps |\partial_t u_\eps|^2 dx dt\right).   
      \end{equation}
      By the energy-dissipation equality \eqref{energy-dissipation equality} the sequence $\eps |\partial_t u_\eps|^2$ is bounded in $L^1$ and thus, along a subsequence, has a weak*-limit $\mu$ as Radon measures.
      In the limit we get, applying Lemma \ref{lem: conv exc mult} along the way,
      \begin{align*}
	\limsup_{\eps\downarrow0}  \left|\int_0^T \int \left(\eps (\xi \cdot \nabla ) u_\eps + \xi\cdot \nu^* \partial_u \phi_i(u_\eps) \right)\cdot \partial_t u_\eps\,dx \,dt \right|  
	\lesssim  \|\xi\|_\infty \left(\frac{1}{\alpha} \exc^{ij}( \nu^*;\eta, u)  + \alpha \mu(\eta) \right).
      \end{align*}
    \textit{Step 2: Passing to the limit in the nonlinear term.}\\
      In the second term on the left-hand side of \eqref{dtu step1} we may now use the chain rule again to see
      \begin{align*}
      	-\int_0^T\int \xi\cdot \nu^* \partial_u \phi_i(u_\eps) \cdot \partial_t u_\eps dx dt = & - \int_0^T\int \xi\cdot \nu^* \partial_t \left(\phi_i \circ u_\eps\right)  dx dt \\
      	& \qquad \to  - \int_0^T\int  \xi \cdot \nu^* \partial_t\left(\phi_i\circ \sum_{1 \leq k \leq P} \chi_k \alpha_k\right) dt.
      	\end{align*}
    \textit{Step 3: Rewriting the limit in terms of the interface between $\chi_i$ and $\chi_j$.}\\
      We can rewrite this limit to read
      \begin{align*}
      	int_0^T\int  \xi \cdot \nu^* \partial_t\left(\phi_i\circ \sum_{1 \leq k \leq P} \chi_k \alpha_k\right) dt &  \overset{\phantom{\ref{prop_normal_velocities}}}{=} - \int_0^T \int \xi \cdot \nu^* \sum_{1\leq k \leq P} \sigma_{ik} \partial_t \chi_k\\
      	& \overset{\ref{prop_normal_velocities}}{=} - \int_0^T \int \xi \cdot \nu^* \sum_{1\leq k \leq P} \sigma_{ik} V_k |\nabla \chi_k| dt.
      \end{align*}
      Thanks to the tilt-excess \eqref{exc mult} we can now get rid of all terms except the $j$-th one:
      With a little help from our friends Cauchy, Schwarz and Young we arrive at
      \begin{align*}
      \left|-\int_0^T \int \xi \cdot \nu^* \sum_{1\leq k \leq P} 
      \sigma_{ik} V_k |\nabla \chi_k|dt + \int_0^T \int \xi \cdot \nu^* \sigma_{ij} V_j \left|\nabla \chi_j\right| dt \right| \\
      \lesssim \|\xi\|_\infty \left(\frac{1}{\alpha}\exc^{ij}( \nu^*;\eta, u) + \alpha  \int_0^T \int \eta \sum_{1\leq k \leq P} V_k^2|\nabla \chi_k|dt\right)
      \end{align*}
      for a smooth cut-off $\eta$ for the support of $\xi$.
      Here, due to the $L^2$-estimate Proposition \ref{prop_normal_velocities}, the right-hand side is an acceptable error term after redefining $\mu$.
      
      Hence we are left with a term only depending on the $j$-th phase which we can replace with (minus) the according term for the $i$-th phase:
      Indeed, using $\sum_k \chi_k =1 $ the error in doing so is equal to
      \begin{align*}
	\left|  \int_0^T \int \xi \cdot \nu^\ast \sigma_{ij} \left( V_j \left| \nabla \chi_j\right| +  V_i \left| \nabla \chi_i\right|\right) dt\right|
	= & \left|  \int_0^T \int \xi \cdot \nu^\ast \sigma_{ij}\partial_t \left(1- \sum_{k\notin\{i,j\}} \chi_k \right) dt\right|\\
    \lesssim &\int_0^T \int |\xi|  \sum_{k\notin\{i,j\}} |V_k| \left|\nabla \chi_k\right|  dt,
      \end{align*}    
      which by Young's inequality is controlled by the same right-hand side as before.
      
      Exploiting $|\nu^\ast -\nu_i| \,|V_i| \lesssim \frac1\alpha |\nu^\ast -\nu_i|^2 + \alpha |V_i| $ we now use the tilt-excess once again to ``un-freeze'' the approximate normal $\nu^*$
      and eliminate other interfaces:
      \begin{align*}
	\left|\int_0^T \int \xi \cdot \nu^* \sigma_{ij} V_i \left|\nabla \chi_i\right| dt - \int_0^T \int \xi \cdot \nu_i \sigma_{ij} V_i \frac12 \left( \left| \nabla\chi_i\right|+\left| \nabla\chi_j\right| - \left| \nabla(\chi_i+\chi_j)\right|\right)dt\right|\\
	\lesssim \|\xi\|_\infty \left(\frac{1}{\alpha}\exc^{ij}( \nu^*;\eta, u) + \alpha \int_0^T \int  \eta\, V_i^2|\nabla \chi_i| dt \right).   
      \end{align*}
      
      Retracing our steps we see that we arrived at the desired estimate.
    \end{proof}
    
    We conclude this section with the proof of our main result.
    \begin{proof}[Proof of Theorem \ref{thm AC2MCF}]
    We found the limit $u$  of the approximations $u_\eps$ in Proposition \ref{prop_comp}, verified the initial conditions in Lemma \ref{Upgraded_compactness} and constructed
    the normal velocity with the according $L^2$-bounds in Proposition \ref{prop_normal_velocities}.
    We only have to prove the motion law (\ref{H=v}).
    Given a smooth test vector field $\xi \in C_0^\infty((0,T)\times \torus,\R^d)$, by Lemma \ref{existence} we may multiply the Allen-Cahn Equation
    \eqref{allen cahn} by  $\eps \left( \xi \cdot \nabla\right) u_\eps$ and integrate w.r.t.\ space and time:
    \begin{equation}\label{allen cahn weak}
      \int_0^T \int \eps (\xi \cdot \nabla )\, u_\eps \cdot \partial_t u_\eps\,dx \,dt 
      =  \int_0^T  \int  \left( \eps \Delta u_\eps -\frac{1}{\eps} \partial_u W (u_\eps)\right) \cdot \left( \xi \cdot \nabla\right) u_\eps\,  dx \,dt.
    \end{equation}
    By Proposition \ref{multi luckhaus modica} the convergence of the energies \eqref{conv ass pw in t} imply the convergence of the first variations for a.e.\ $t$.
    Recall that by \eqref{reshetnyak PI gives estimate} and Lebesgue's dominated convergence the right-hand side of \eqref{allen cahn weak} converges:
    \begin{align*}
    &\lim_{\eps\downarrow0} \int_0^T \int \left( \eps \Delta u_\eps -\frac{1}{\eps} \partial_u W (u_\eps)\right) \cdot \left( \xi \cdot \nabla\right) u_\eps\, dx \,dt\\
    &= \sum_{i,j} \sigma_{ij} \int_0^T \int \nabla \xi \colon 
    \left( Id - \nu_i \otimes \nu_i\right) \frac12 \left( \left| \nabla \chi_i \right| + \left| \nabla \chi_j \right| - \left| \nabla (\chi_i+\chi_j) \right|\right) dt.
    \end{align*}
    In order to prove the convergence of the left-hand side, we proceed as in \cite{laux2015convergence}.
    We decompose $\xi = \sum_{B\in \B_r} \varphi_B \xi$ with a partition of unity underlying the covering $\B_r$ defined in \eqref{covering Br}.
    Using Proposition \ref{multi prop dt nu} for $\xi_B = \varphi_B \xi$ on time intervals $0=T_1< \ldots < T_K = T$ and passing to the limit $K\to \infty$ we obtain
    the error
    \begin{align*}
     &\left| \int_0^T \int \eps (\xi \cdot \nabla ) u_\eps \cdot \partial_t u_\eps\,dx \,dt 
     - \sum_{1\leq i,j\leq P} \sigma_{ij} \int_0^T \int  V_i\, \xi \cdot \nu_i \frac{1}{2}\left(|\nabla \chi_i| + |\nabla \chi_j| - |\nabla (\chi_i + \chi_j)|\right) dt\right| \\
     \lesssim & \, \|\xi\|_\infty \Biggr(\frac{1}{\alpha} \int_0^T \sum_{B\in \B_r} \min_{i,j}\min_{\nu^\ast \in \sph^{d-1}} \int \eta_B \left|\nu_i - \nu^\ast\right|^2 \left| \nabla \chi_i \right| + \int \eta_B \left|\nu_j + \nu^\ast\right|^2 \left| \nabla \chi_j \right| \\
      & \qquad \qquad \qquad \qquad \qquad \qquad \qquad  +  \sum_{k \notin \{i,j\}} \int \eta_B  \left| \nabla \chi_k\right| dt  + \alpha \int_0^T \int \sum_{B\in\B_r}\eta_B \,d\mu \Biggr),
    \end{align*}
    where for a ball $B$ the function $\eta_B$ denotes a cutoff for $B$ in $2B$ as in equation \eqref{baldo E is supremum}.
    Because of the finite overlap \eqref{finiteoverlap} the last term is uniformly bounded in $r$.
    Using Lemma \ref{lem_approximate_normal_L2} we see that the first term vanishes as $r\to 0$.
    Then taking $\alpha\to0$ we obtain the convergence of the velocity-term and thus verified the motion law \eqref{H=v}.
    \end{proof}

  \section{Forces and volume constraint}\label{sec:forces,volume}
  The proofs in Section \ref{sec:comp} and Section \ref{sec:conv} stem from the a priori estimate \eqref{energy-dissipation equality} and the convergence assumption \eqref{conv_ass}.
  We mostly used the Allen-Cahn Equation \eqref{allen cahn} to prove this a priori bound.
  Besides that we made use of it only at one other point, in the proof of Theorem \ref{thm AC2MCF} in the form of \eqref{allen cahn weak} 
  and the justification for testing the equation with $\eps(\xi\cdot\nabla)u_\eps$.
  
  In this section we exploit this flexibility of our proof and apply it to the case when external forces are present or when a volume constraint is active, cf.\ Theorem \ref{thm FAC2FMCF}
  and Theorem \ref{thm VPAC2VPMCF}, respectively.
  
  \subsection{External forces}
  
  Since the forces $f_\eps$ in equation \eqref{allen cahn f} come from an extra energy-term we do not expect to have the same energy-dissipation equality as in the case above where $f_\eps\equiv 0$.
  Indeed, one can view \eqref{allen cahn f} as the (again by the factor $\frac1\eps$ accelerated) $L^2$-gradient flow of the total energy
  \[
    E_\eps(u_\eps) + \int f_\eps \cdot u\,dx,
  \]
  which is the sum of the ``surface energy'' $E_\eps(u_\eps)$ and the ``bulk energy'' $\int f_\eps \cdot u\,dx$.
  Since the extra term is a compact perturbation in the static setting, these total energies $\Gamma$-converge to
  \[
   E(u) + \int f\cdot u\,dx.
  \]
  This energetic view-point seems also the most natural way to understand the scaling in $\eps$ for the forces $f_\eps$ in equation \eqref{allen cahn f}.
  Under our assumption on the forces $f_\eps$ in Theorem \ref{thm FAC2FMCF} we can control this bulk energy and 
  get an \emph{estimate}  on the ``surface energy'' $E_\eps(u_\eps)$ and the dissipation, which is reminiscent of equality \eqref{energy-dissipation equality}.
  \begin{lem}\label{lem energy-dissipation f}
   Let $u_\eps$ solve the forced Allen-Cahn Equation \eqref{allen cahn f}. Then
   \[
    E_\eps(u_\eps(T)) + \int_0^T \int \eps \left|\partial_t u_\eps \right|^2 dx\,dt 
    \leq E_\eps(u_\eps(0)) + C\, e^{\eps\,C\sqrt{T} \left(\int_0^T \int \left|\partial_t f_\eps\right|^2dx\,dt\right)^\frac12}\sqrt{T} \left(\int_0^T \int \left|\partial_t f_\eps\right|^2dx\,dt\right)^\frac12.
   \]
  \end{lem}
  
 \begin{proof}[Proof of Lemma \ref{lem energy-dissipation f}]
   We differentiate the energy $E_\eps$ along the trajectory of $t\mapsto u_\eps(t)$ and integrate by parts
   \begin{align*}
     \frac{d}{dt} E_\eps(u_\eps) 
	= & \int \eps \nabla u_\eps : \nabla \partial_t u_\eps + \frac1\eps \partial_u W(u_\eps) \cdot \partial_t u_\eps \,dx\\
	= & \int  \eps \left( - \Delta u_\eps + \frac1{\eps^2} \partial_u W(u_\eps) \right) \cdot \partial_t u_\eps \,dx\\
	\overset{\eqref{allen cahn f}}{=} & -\int \eps |\partial_t u_\eps |^2 \, dx+ \int f_\eps \cdot \partial_t u_\eps\,dx.
   \end{align*}
  We integrate from $0$ to $T$ and obtain
  \begin{equation}\label{proof of energy est f}
   E_\eps(u_\eps(T)) + \int_0^T \int \eps \left|\partial_t u_\eps \right|^2 dx\,dt = E_\eps(u_\eps(0)) + \int_0^T \int f_\eps \cdot \partial_t u_\eps \,dx\,dt.
  \end{equation}
  Now we integrate the last integral by parts and obtain by Cauchy-Schwarz
  \[
   \left|\int_0^T \int f_\eps \cdot \partial_t u_\eps \,dx\,dt\right| \leq
   \int_0^T\left( \int \left|\partial_t f_\eps\right|^2dx\right)^\frac12 \left( \int \left|u_\eps\right|^2dx\right)^\frac12dt.
  \]
  The coercivity assumption \eqref{growth} of $W$ at infinity yields a bound for the second factor:
  \[
   \left(\int \left|u_\eps\right|^2dx\right)^\frac12\lesssim  \left( 1+\eps E_\eps(u_\eps)\right)^\frac12 \lesssim 1+\eps E_\eps(u_\eps).
  \]
  A Gronwall argument helps us out: 
  \[
    E_\eps(u_\eps(T)) + \int_0^T \int \eps \left|\partial_t u_\eps \right|^2 dx\,dt \leq E_\eps(u_\eps(0)) + 
    C\, e^{\eps\, C\int_0^T \left(\int |\partial_t f_\eps |^2 dx\right)^\frac12dt}\int_0^T \left(\int |\partial_t f_\eps|^2 dx\right)^\frac12 dt,
  \]
  which yields the claim. Note that by our assumption on the forces the exponential prefactor is $\sim 1$ for small $\eps$.
  \end{proof}
  This estimate is indeed enough to apply our techniques to the case of \eqref{allen cahn f}.
  \begin{proof}[Proof of Theorem \ref{thm FAC2FMCF}]
   As noted in Remark \ref{compactness_for_forces}, the a priori estimate, Lemma \ref{lem energy-dissipation f}, allows us to apply the statements in
   Section \ref{sec:comp} so that in particular 
   we can find a convergent subsequence $u_\eps\to u$ satisfying the initial conditions by Lemma \ref{Upgraded_compactness}, for some $u=\sum_i \chi_i \alpha_i$, and we can construct the normal velocities under the 
   convergence assumption \eqref{conv_ass}. The  bounds for $f_\eps$ allow us to extract a further subsequence
   such that also the forces converge to some $f\in H^1((0,T)\times \torus,\R^N)$:
   \begin{equation}\label{f eps to f}
    f_\eps \to f \quad\text{in } L^2\quad \text{and}\quad \nabla f_\eps \rightharpoonup \nabla f\quad \text{in }L^2.
   \end{equation}
  
   If we formally differentiate the equation \eqref{allen cahn f} and use $\nabla f_\eps \in L^2$ we can show as in Step 2 of the proof of Lemma \ref{existence} that $\partial_i\partial_j u_\eps, \partial_u W(u_\eps)\in L^2$.
   Hence we are allowed to test the equation for $u_\eps$, here the forced Allen-Cahn Equation \eqref{allen cahn f}, 
   with $\eps \left(\xi \cdot \nabla \right) u_\eps$ to obtain
   \[
    \int_0^T \int \eps (\xi \cdot \nabla )\, u_\eps \cdot \partial_t u_\eps\,dx \,dt
    =\int_0^T  \int \left( \eps \Delta u_\eps -\frac{1}{\eps} \partial_u W (u_\eps)\right) \cdot \left( \xi \cdot \nabla\right) u_\eps + f_\eps \cdot \left(\xi \cdot \nabla\right) u_\eps \, dx \,dt.
   \]
   Integrating the last term by parts gives
   \[
     \int_0^T \int f_\eps \cdot \left(\xi \cdot \nabla\right) u_\eps \, dx\,dt = - \int_0^T \int  \left(\nabla \cdot \xi\right)f_\eps\cdot u_\eps +\left(\xi \cdot \nabla\right)f_\eps \cdot u_\eps \, dx\,dt.
   \]
   Since $u_\eps \to u=\sum_i \chi_i \alpha_i$ in $L^2$ and \eqref{f eps to f} we can pass to the limit $\eps\to0$ and obtain
   \[
    \int_0^T \int  \left(\nabla \cdot \xi\right)f\cdot u +\left(\xi \cdot \nabla\right)f \cdot u \, dx\,dt
    = \sum_{i=1}^N \int_0^T \int \left(f \cdot \alpha_i\right) \left( \xi \cdot \nu_i\right) \left|\nabla \chi_i\right|dt.
   \]
   We can apply Proposition \ref{multi luckhaus modica} to pass to the limit in the curvature-term. For the velocity-term we may apply Proposition \ref{multi prop dt nu} 
   and follow the lines of the proof of Theorem \ref{thm AC2MCF} for the localization argument. We thus verified \eqref{H=v+f}.
  \end{proof}
  
  \subsection{Volume constraint}
  Again, our starting point is an energy-dissipation estimate.
  It is quite natural that the solution of the volume-preserving Allen-Cahn Equation \eqref{allen cahn vol} satisfies the same energy-dissipation equation as the solution of the Allen-Cahn Equation \eqref{allen cahn}.
  \begin{lem}\label{lem energy dissipation vol}
   Let $u_\eps$ solve the volume-preserving Allen-Cahn Equation \eqref{allen cahn vol}. Then
   \begin{equation}
    E_\eps(u_\eps(T)) + \int_0^T \int \eps \left|\partial_t u_\eps \right|^2 dx\,dt = E_\eps(u_\eps(0)).    \label{energy dissipation estimate vol}
  \end{equation}
  \end{lem}
  
  \begin{proof}[Proof of Lemma \ref{lem energy dissipation vol}]
    We follow the lines of the proof of Lemma \ref{lem energy-dissipation f} until \eqref{proof of energy est f} with $f_\eps(x,t)$ replaced by $\lambda_\eps(t)$. Since $\lambda_\eps$ is independent of $x$ 
    for the second right-hand side integral in \eqref{proof of energy est f} we have
    \[
      \int_0^T \int \lambda_\eps \,\partial_t u_\eps\,dx\,dt = \int_0^T \lambda_\eps \frac{d}{dt} \int u_\eps \,dx\,dt.
    \]
    But by the choice of $\lambda_\eps$ integrating \eqref{allen cahn vol} gives $\frac{d}{dt}\int u_\eps\,dx = 0$ and we obtain \eqref{energy dissipation estimate vol}.
  \end{proof}

  \begin{proof}[Proof of Theorem \ref{thm VPAC2VPMCF}]
  Since we have the same energy-dissipation estimate, Lemma \ref{lem energy dissipation vol}, as in the unconstrained case, by Remark \ref{compactness_for_forces} we can apply the statements
  in Section \ref{sec:comp} so that in particular we obtain a convergent subsequence $u_\eps\to u$  as before and we can construct the normal velocities under the convergence assumption \eqref{conv_ass}.
  
  The Lagrange multiplier $\lambda_\eps$ does not depend on the space variable $x$ and hence the same computation as in Step 2 in the proof of
  Lemma \ref{existence} yields $\partial_i\partial_j u_\eps,\partial_u W(u_\eps)\in L^2$ and we may test our equation \eqref{allen cahn vol} with $\eps \left(\xi \cdot \nabla\right) u_\eps$
  and obtain
  \[
     \int_0^T \int \eps (\xi \cdot \nabla )\, u_\eps \cdot \partial_t u_\eps\,dx \,dt
     =\int_0^T  \int \left( \eps \Delta u_\eps -\frac{1}{\eps} \partial_u W (u_\eps)\right) \cdot \left( \xi \cdot \nabla\right) u_\eps\, dx \,dt 
     +\int_0^T \lambda_\eps \int \left(\nabla \cdot \xi \right) u_\eps\, dx\,dt.
    \]
  We wish to pass to the limit in this weak formulation of \eqref{allen cahn vol}.
  
  By Proposition \ref{multi luckhaus modica} we can pass to the limit in the first right-hand side term and the left-hand side term.
  Again, with Proposition \ref{multi prop dt nu} and the localization argument in the proof of Theorem \ref{thm AC2MCF} we can pass to the limit on the left-hand side.
  In order to pass to the limit in the second right-hand side term we use Proposition \ref{prop estimate lambda} below, which provides control of $\lambda_\eps$ in $L^2$.
  After passage to a further subsequence if necessary we have
  \begin{equation*}
   \lambda_\eps \rightharpoonup \lambda \quad \text{weakly in } L^2(0,T)
  \end{equation*}
  and since by Lemma \ref{Upgraded_compactness}
  \[
   \int \left(\nabla \cdot \xi \right) u_\eps\, dx \to \int \left(\nabla \cdot \xi \right) u\, dx\quad \text{strongly in } L^2(0,T)
  \]
  we can pass to the limit in the product. This concludes the proof of the theorem.
  \end{proof}

  \begin{prop}[Estimates on Lagrange multiplier]\label{prop estimate lambda}
    Let $u_\eps$ solve \eqref{allen cahn vol} and let $\lambda_\eps$ be the Lagrange multiplier \eqref{def lambda}. Then 
    \[
      \limsup_{\eps\to0} \int_0^T \lambda_\eps^2 \, dt \lesssim (1+T) E_0.
    \]
  \end{prop}
  
  \begin{proof}[Proof of Proposition \ref{prop estimate lambda}]
    We follow the idea of the proof of Proposition 1.12 in \cite{LauSwa15}. 
    For a given test vector field $\xi\in C_0^\infty((0,T)\times\torus,\R^d)$ we first multiply \eqref{allen cahn vol} by 
    $\eps \left(\xi \cdot \nabla\right) u_\eps$,  integrate in space and take the square:
    \begin{equation}
      \lambda_\eps^2 \left(\int \left(\nabla \cdot \xi \right) u_\eps\, dx\right)^2 \lesssim \left(\int \left( \eps \Delta u_\eps -\frac{1}{\eps} \partial_u W (u_\eps)\right) \cdot \left( \xi \cdot \nabla\right) u_\eps\, dx\right)^2
      + \left(\int \eps (\xi \cdot \nabla )\, u_\eps \cdot \partial_t u_\eps\,dx\right)^2.
    \end{equation}
    With Cauchy-Schwarz we can estimate the second right-hand side term
    \[
      \left(\int \eps (\xi \cdot \nabla )\, u_\eps \cdot \partial_t u_\eps\,dx\right)^2 \lesssim  \|\xi\|_\infty^2 \left(\eps \int\left| \partial_t u_\eps \right|^2 dx\right) E_\eps(u_\eps).
    \]
    For the first right-hand side term we use \eqref{reshetnyak PI gives estimate} to obtain
    \[
      \left(\int \left( \eps \Delta u_\eps -\frac{1}{\eps} \partial_u W (u_\eps)\right) \cdot \left( \xi \cdot \nabla\right) u_\eps\, dx\right)^2
      \lesssim \|\nabla \xi\|_\infty^2 E_\eps(u_\eps)^2.
    \]
    Since $\nabla\cdot \xi$ is orthogonal to constant functions we might subtract the average $\langle u_\eps \rangle :=\fint u_\eps\, dx$ of $u_\eps$ on the left-hand side and obtain
    \[
      \lambda_\eps^2 \left(\int \left(\nabla \cdot \xi \right) \left(u_\eps-\langle u_\eps\rangle\right) dx\right)^2 
      \lesssim \|\nabla \xi\|_\infty^2 E_\eps(u_\eps)^2 +\|\xi\|_\infty^2 \left(\eps \int\left| \partial_t u_\eps \right|^2 dx\right) E_\eps(u_\eps).
    \]
    We integrate in time and apply the energy-dissipation estimate \eqref{lem energy dissipation vol} on the right-hand side:
    \begin{equation*}
      \int_0^T \lambda_\eps^2 \left(\int \left(\nabla \cdot \xi \right) \left(u_\eps-\langle u_\eps\rangle\right) dx\right)^2 dt \lesssim \sup_t  \|\xi\|_{W^{1,\infty}}^2 \left(1+T\right) E_0^2.
    \end{equation*}
    
    Hence it is enough to find a test field $\xi$ such that we can bound the left-hand side integral from below while the right-hand side stays uniformly bounded:
    \begin{align}
    \fint \left(\nabla \cdot \xi \right) \left(u_\eps-\langle u_\eps\rangle\right) dx &\geq \frac12\quad \text{and}\label{xi1}\\
     \|\xi\|_{W^{1,\infty}} &\lesssim 1+  E_0.\label{xi2}
    \end{align}
    
    We now proceed by constructing a vector field $\xi$ satisfying \eqref{xi1} and \eqref{xi2} in a similar manner as in \cite{LauSwa15}.
    To this end we first fix some $t\in(0,T)$, convolve the limit $u=\lim u_\eps$  with a standard mollifier 
    $\varphi_\delta(x)= \frac1{\delta^d}\varphi(\frac x\delta)$ on scale $\delta>0$ (to be chosen later) and write $u_\delta:= \varphi_\delta \ast u$. Then we let $v\colon \torus \to \R$ denote the solution of
    \begin{equation}\label{potential for xi}
     \Delta v = \varphi_\delta \ast\left(u -\langle u \rangle\right) = u_\delta -\langle u \rangle.
    \end{equation}
    Note that since the right-hand side has vanishing integral, this problem is well-posed.
    We set $\xi := \nabla v$ and verify \eqref{xi1} which works by construction of $\xi$ and \eqref{xi2} which boils down to elliptic estimates.
    
    \textit{Step 1: Argument for the lower bound \eqref{xi1}.}\\
    By Lemma \ref{Upgraded_compactness} we have $u_\eps\to u$ in $C_tL^2_x$ as $\eps \to 0$. Thus
    \begin{align*}
     \inf_t \int \left(\nabla \cdot \xi \right) \left(u_\eps-\langle u_\eps\rangle\right) dx 
     = & \inf_t \left\{ \int\left(u-\langle u\rangle\right)^2 dx + \int  \left(u_{\delta} -u \right) \left(u - \langle u \rangle \right) dx\right\} +o(1)\quad \text{as }\eps\to0.
    \end{align*}
    Since $u=\sum_i \chi_i \alpha_i$ we have for the first left-hand side integral
    \[
     \int\left(u-\langle u\rangle\right)^2 dx = \int\left(u-\langle u^0\rangle\right)^2 dx 
     \geq \textup{dist}\left(\langle u^0\rangle, \{\alpha_1,\dots,\alpha_\numphases\}\right)^2 \Lambda^d.
    \]
    The second left-hand side integral can be estimated with help of the energy \eqref{E}:
    \[
      \left| \int  \left(u_{\delta} -u \right) u\, dx \right| \lesssim \int \left|u_\delta-u \right| dx \leq \delta \int \left|\nabla u \right| \lesssim \delta\, E(u) \leq \delta\, E_0.
    \] 
    Setting $\delta:= \frac1C \frac1{E_0}\textup{dist}\left(\langle u^0\rangle, \{\alpha_1,\dots,\alpha_\numphases\}\right)^2 \Lambda^d >0$  for some sufficiently large constant $C<\infty$,
    we arrive at \eqref{xi1} for sufficiently small $\eps$.
    
    \textit{Step 2: Argument for the estimate \eqref{xi2}.}\\
    The upper bound \eqref{xi2} follows from basic elliptic regularity theory. We fix some exponent $q=q(d)>d$.  Since $u=\sum_i \chi_i \alpha_i$ is uniformly bounded, the Calder\'on-Zygmund inequality yields
    \[
      \int \left| \nabla\xi\right|^q dx\lesssim \int \left|u_{\delta} - \langle u_{\delta} \rangle \right|^q dx \lesssim 1.
    \]
    Since the right-hand side is smooth, we can differentiate the equation \eqref{potential for xi} for $v$ and obtain:
    \[
     \Delta \xi = \nabla u_{\delta}
    \]
    and we obtain again by Calder\'on-Zygmund
    \[
     \left( \int \left| \nabla^2\xi\right|^q dx\right)^{\frac1q}\lesssim\left( \int \left|\nabla u_{\delta}\right|^q dx\right)^{\frac1q} \lesssim \int \left|\nabla \varphi_\delta\right|dx \lesssim \frac{1}{\delta}.
    \]
    Since $\langle \xi\rangle =0$ we thus have by Poincar\'e's inequality
    $
     \|\xi\|_{W^{2,q}} \lesssim \frac{1}{\delta}
    $
    and since $q>d$  Morrey's inequality yields
    \[
      \|\xi\|_{W^{1,\infty}} \lesssim 1+\frac{1}{\delta} \sim 1+ E_0,
    \]
    which is precisely our claim \eqref{xi2}.
  \end{proof}

    \printbibliography

    \end{document}